\documentclass[12pt,twoside,a4paper,leqno]{article}

\usepackage[utf8]{inputenc}
\usepackage[T1]{fontenc} 
\usepackage[english]{babel} 
\usepackage{microtype} 
\usepackage{geometry}
\usepackage{lmodern}            

\usepackage{amsmath,amsfonts,amssymb,amsthm}
\usepackage{mathtools} 

\usepackage{bm}      
\usepackage{dsfont}  
\usepackage{bbm}   
\usepackage{mathrsfs}
\usepackage[bb=px]{mathalfa}  

\usepackage[dvipsnames]{xcolor}
\usepackage{graphicx}
\usepackage{booktabs}  
\usepackage{enumitem}
\setlist[enumerate]{itemsep=0pt}
\usepackage{fancyhdr}
\usepackage{titlesec}
\usepackage{authblk}    
\usepackage{float}
\usepackage{siunitx} 

\usepackage{colortbl}

\mathtoolsset{showonlyrefs}
\usepackage[square,numbers,sort]{natbib}
\usepackage[ruled,linesnumbered]{algorithm2e}

\usepackage{tikz}
\usetikzlibrary{positioning,fit,tikzmark, calc,decorations.pathreplacing, calligraphy,arrows.meta}

\colorlet{inlinkcolor}{purple}
\colorlet{exlinkcolor}{purple}
\colorlet{citecolor}{teal!50!blue}

\usepackage[colorlinks=true,linkcolor=inlinkcolor,citecolor=citecolor,urlcolor=exlinkcolor,linktoc=page,breaklinks=true,plainpages=false,pagebackref=true,backref=page]{hyperref}
\renewcommand*{\backrefalt}[4]{%
\ifcase #1 \footnotesize{(Not cited.)}%
\or        \footnotesize{(Cited on page~#2.)}%
\else      \footnotesize{(Cited on pages~#2.)}%
\fi}

\newtheorem{theorem}{Theorem}[section]
\newtheorem{proposition}[theorem]{Proposition}
\newtheorem{lemma}[theorem]{Lemma}

\theoremstyle{definition}

\theoremstyle{remark}
\newtheorem{remark}[theorem]{Remark}

\numberwithin{equation}{section} 

\newcommand{\setheader}[2]{%
\pagestyle{fancy}
\fancyhf{}
\fancyhead[LO]{#1}
\fancyhead[RO]{}
\fancyhead[LE]{}
\fancyhead[RE]{#2}
\renewcommand{\headrulewidth}{0.4pt}
\fancyfoot[C]{\thepage}
\renewcommand{\headrule}{\color{darkgray!50}\hrule width\textwidth height 1.0pt}
}

\makeatletter
\renewcommand{\maketitle}{
\begin{center}
{\Large\bfseries\@title\par}
\vskip 1.5em
{\@author\par}
\vskip 1em
{\@date\par}
\end{center}
\vskip 1.5em
}
\makeatother

\makeatletter
\renewenvironment{abstract}{
\small\noindent{\bfseries Abstract.}
\ignorespaces
}{\par\vspace{0.5em}}
\makeatother

\newcommand{\keywords}[1]{
\noindent{\small\bfseries Keywords:} {\small #1}
\par\noindent\textcolor{gray}{\rule{\linewidth}{0.5pt}}
}

\titleformat{\section}
{\normalfont\large\bfseries\centering}
{\thesection}{1em}{}
\titleformat{\subsection}[runin]
{\normalfont\normalsize\bfseries}
{\thesubsection}{1em}{}[.]
\titleformat{\subsubsection}[runin]
{\normalfont\normalsize\itshape}
{\thesubsubsection}{1em}{}[.]
\titleformat{\paragraph}[runin]
{\normalfont\normalsize\bfseries}
{\theparagraph}{1em}{}[.]

\titlespacing{\section}
{0pt}{2.0ex plus 0.5ex minus 0.2ex}{1.0ex plus 0.1ex}
\titlespacing{\subsection}
{0pt}{1.0ex plus 0.1ex minus 0.1ex}{1.0ex plus 0.1ex}
\titlespacing{\subsubsection}
{0pt}{1.0ex plus 0.1ex minus 0.1ex}{1.0ex plus 0.1ex}
\titlespacing{\paragraph}
{0pt}{0.5ex plus 0.1ex minus 0.1ex}{0.5ex plus 0.1ex}

\DeclareMathOperator*{\kl}{KL}
\DeclareMathOperator{\fr}{FR}
\DeclareMathOperator{\wfr}{WFR}

\DeclareMathOperator*{\argmin}{arg\, min}
\renewcommand{\div}{\,\mathrm{div}\,}
\renewcommand{\d}{\,\mathrm{d}}
\newcommand{\R}{\mathbb{R}}
\newcommand{\weakly}{\rightharpoonup}\normalfont
\newcommand{\best}{\cellcolor{green!15!white}\bf }
\newcommand{\second}{\cellcolor{blue!15!white}\bf }

\setheader{Sampling via Wasserstein--Fisher--Rao JKO Scheme}{}

\begin{document}
\title{Neural Sampling with Reweighted Normalizing Flows via the Wasserstein--Fisher--Rao JKO Scheme}
\author[1]{Chenguang Duan} 
\author[2]{Johannes Hertrich}
\author[3]{Gabriele Steidl}

\affil[1]{Institut f{\"u}r Geometrie und Praktische Mathematik, RWTH Aachen University, \texttt{duan@igpm.rwth-aachen.de}}
\affil[2]{Institute of Computer Science, University of G\"ottingen, \texttt{johannes.hertrich@uni-goettingen.de}}
\affil[3]{Institute of Mathematics, TU Berlin, \texttt{steidl@math.tu-berlin.de}}
\date{\today}

\thispagestyle{plain}
\maketitle

\begin{abstract}
    We propose a neural algorithm for sampling from distributions specified by unnormalized Boltzmann densities. Our approach is based on the Jordan--Kinderlehrer--Otto scheme for the Kullback--Leibler divergence in the Wasserstein--Fisher--Rao geometry (WFR JKO scheme). Our contributions are twofold. First, we prove that, for any fixed step size, the exact WFR JKO iterates converge exponentially fast to the target as the number of iterations tends to infinity. Notably, this result requires no structural assumptions on the target, such as log-concavity or a logarithmic Sobolev inequality. Second, we develop a neural implementation of the WFR JKO scheme that parametrizes its transport and reaction components using reweighted normalizing flows. Numerical experiments on challenging multimodal targets demonstrate the promising performance of the proposed method.
\end{abstract}

\keywords{Wasserstein--Fisher--Rao distance, Hellinger--Kantorvich distance, Sampling from Boltzmann densities, JKO scheme, Metropolis--Hastings algorithm}

\section{Introduction}

\par Let $\pi(x)=\exp(-V(x))$ be a strictly positive, integrable Boltzmann density on $\mathbb{R}^{d}$. We identify $\pi$ with the corresponding finite measure and denote its total mass by $Z\coloneqq\pi(\mathbb{R}^{d})$. Our aim is to generate samples from the probability distribution $\tilde\pi=\pi/Z$, using the unnormalized density $\pi$ without prior knowledge of $Z$. Such sampling problems arise in Bayesian inverse problems and data assimilation~\cite{Durmus2019High,Overstall2020Bayesian, Heng2021Gibbs,martin2025pnpflow,chang2025provable,Durmus2025Generative,cardoso2024monte,Ding2026Nonlinear}, statistical mechanics~\cite{Landau2014Guide,Frank2019Boltzmann,midgley2023flow}, and machine learning~\cite{Cobb2021Scaling,Izmailov2021what}. They are particularly challenging in high dimensions and when the target has several well-separated modes. In these settings, accurate sampling requires both exploration of the relevant regions of the state space and recovery of their relative probability masses.

\par A variational approach to this problem is to formulate sampling as the minimization of a discrepancy functional over a space of measures. A natural objective is the Kullback--Leibler (KL) divergence, whose unique minimizer is the target distribution. Equipping the space of measures with a metric determines a gradient flow of this objective and thereby provides a starting point for constructing sampling algorithms. In particular, the gradient flow of $\kl(\cdot\,\|\,\tilde\pi)$ in the quadratic Wasserstein geometry is the Fokker--Planck equation associated with overdamped Langevin dynamics~\cite{JKO1998}. The law of the Langevin diffusion thus decreases the KL divergence through the transport of probability mass.

\par This common continuous-time description leads to two discretization routes. On the particle level, an explicit Euler--Maruyama discretization of the Langevin diffusion yields the unadjusted Langevin algorithm~\cite{Durmus2019High}, with stochastic-gradient variants supporting applications involving large datasets~\cite{WT2011}. On the level of measures, the Jordan--Kinderlehrer--Otto (JKO) scheme~\cite{JKO1998} provides an implicit variational discretization. Each step minimizes the KL divergence together with a squared Wasserstein-distance penalty relative to the preceding iterate. These variational problems can be approximated using normalizing flows, leading to neural JKO sampling methods~\cite{Hertrich2025Importance}. Continuous normalizing flows are particularly suitable for this purpose, since they describe both transport and density evolution through ordinary differential equations. Related implementations of the Wasserstein JKO scheme have also been developed for generative modeling~\cite{OFLR2021,VWTON2023,Xu2023Normalizing}.

\par Despite their different numerical realizations, both approaches retain the transport geometry of the underlying Wasserstein gradient flow. For multimodal targets, this geometry can lead to slow redistribution of mass: transferring probability between separated modes requires crossing regions of low density. The resulting metastability is reflected in the convergence analysis. For nonconvex potentials, the KL divergence need not be geodesically convex in the Wasserstein metric~\cite{Chewi2025Statistical}, and quantitative convergence estimates commonly rely on functional inequalities, such as a logarithmic Sobolev inequality (LSI)~\cite{Bakry1985Diffusions,Vempala2019Rapid,CELSZ2025}. The associated constants can deteriorate substantially as the modes become more strongly separated. Consequently, changing from an explicit to an implicit discretization does not, by itself, remove the difficulty of sampling multimodal targets. This motivates augmenting local transport with mechanisms for global redistribution, as in the reweighting and resampling steps of sequential Monte Carlo~\cite{DDJ2006}, the learned nonlocal proposals of adaptive Monte Carlo~\cite{GRV2022}, and the importance corrections of neural JKO sampling~\cite{Hertrich2025Importance}.

\par The Wasserstein--Fisher--Rao (WFR) metric, also known as the Hellinger--Kantorovich metric, incorporates transport and mass adjustment into a single geometric framework~\cite{Chizat2018Interpolating,liero2016optimal,liero2018optimal}. Its Wasserstein component moves mass through the state space, while its Fisher--Rao component allows mass to be created or removed. The relative masses of separated modes can therefore be adjusted without transporting all of the redistributed mass through the intervening low-density regions. Transport and reaction play complementary roles: the former supports spatial exploration, whereas the latter corrects the allocation of mass. Beyond this geometric intuition, it has been shown that exponential convergence of the WFR gradient flow can be established without assuming log-concavity or a logarithmic Sobolev inequality for the target, provided suitable conditions on the initial density ratio hold~\cite{lu2023birth,crucinio2026sequential}. These results demonstrate that reaction can yield convergence rates independent of potential barriers and provide a theoretical motivation for using WFR-type geometries in multimodal sampling. 

\par An explicit particle realization of this principle leads to Langevin sampling with birth--death mechanisms~\cite{lu2019accelerating,lu2023birth,Tan2025Accelerate}. In the probability-preserving setting, the underlying dynamics correspond to a spherical counterpart of the WFR geometry. The resulting algorithms combine Langevin updates with particle duplication and removal according to the discrepancy between the current and target densities. Although these updates can accelerate redistribution among modes, their implementation introduces a further difficulty: the birth--death rates depend on the evolving density, which is not directly available from an empirical particle measure. Kernel-based approximations introduce bandwidth dependence and become difficult to resolve accurately in high dimensions~\cite{lu2023birth}. Thus, exploiting the favorable properties of the continuum dynamics requires a numerical representation that also provides access to the evolving density.

\par These considerations motivate an implicit variational approach combined with a tractable density representation. In this paper, we develop a neural sampling method based on the WFR JKO scheme. We work with the unbalanced WFR geometry on finite positive measures and minimize the extended KL divergence $\kl(\cdot\,\|\,\pi)$. Its minimizer is the unnormalized measure $\pi$, so the scheme approximates both the shape and the total mass of the target. The unknown normalizing constant appears only as an additive term in the objective and is therefore unnecessary for computing the variational updates. WFR gradient flows and their variational approximations have been studied in~\cite{gallouet2017jko,fleissner2021minimizing,mielke2025hellinger}; here, we analyze the long-time behavior of the discrete scheme at a fixed time step.

\par Our main theoretical result establishes geometric decay of the KL divergence along the exact WFR-JKO iterates. The contraction factor depends only on the time step and the WFR parameter, independently of the target measure. In particular, the estimate requires neither convexity of the potential nor a logarithmic Sobolev inequality for the target. 

\par To make the variational steps computationally accessible, we parametrize both the transport velocity and the reaction rate by neural networks. The characteristic representation of the continuity equation with reaction then combines a continuous normalizing flow with a positive, learned reweighting function. We derive a computable training objective, a Monte Carlo approximation, and a stochastic trace estimator for the divergence term. This representation propagates particle positions, density values, and total mass, thereby avoiding the kernel density estimation required by standard birth--death particle implementations. 

\par The practical algorithm combines these learned transport--reaction steps with resampling to obtain an unweighted particle ensemble. To mitigate the loss of diversity caused by particle duplication, we incorporate Metropolis--Hastings rejuvenation steps targeting the normalized current model distribution~\cite{Hastings1970,Hertrich2025Importance}. The availability of model density evaluations allows proposals in both the data space and the latent space of the learned flow. Numerical experiments on multimodal targets compare the proposed method with established MCMC and neural sampling approaches and examine the effect of these rejuvenation strategies.

\subsection*{Contributions}

\par The main contributions of this work are as follows:
\begin{enumerate}
\item We prove geometric decay of the extended KL divergence along the exact WFR JKO iterates for strictly positive, integrable targets. For each fixed time step, the contraction factor depends only on the time step and the WFR parameter, without requiring log-concavity or a logarithmic Sobolev inequality for the target. 
\item We derive a computable WFR JKO objective by combining continuous normalizing flows with learned reaction rates. This representation jointly describes spatial transport and mass adjustment, provides access to density values and total mass, and avoids kernel density estimation. 
\item We combine the learned transport--reaction updates with resampling and Metropolis--Hastings rejuvenation to mitigate particle degeneracy. We consider proposals in both the data and latent spaces and evaluate the resulting method on multimodal targets through comparisons with established MCMC and neural sampling approaches.
\end{enumerate}

\subsection*{Outline of the paper}  
In Section \ref{sec:prelim}, we collect basic results
on the Wasserstein metric for probability measures as well as on the Fisher--Rao metric 
and the WFR  metric, also known as Hellinger--Kantorivich metric, for positive measures. In particular, we rely on their dynamic formulations to formulate our loss function later in \eqref{eq:obj:WFR:vg}.
In Section \ref{section:WFR:JKO}, we deal with the WFR--JKO scheme for minimizing the KL functional on the space of positive measures. As main result, we prove the exponential convergence of the scheme, see Theorem \ref{thm:main}.
In Section \ref{sec_norm_glows}, we reformulate the  WFR--JKO step minimization 
for an appropriate admissibility class of velocities and mobility functions. We provide
a computable formulation for our objective based on ideas from the normalizing flows  and present the corresponding algorithm.
Since the resampling step unavoidably causes particle degeneracy we propose to rejuvenate the ensemble while preserving its target by applying a Metropolis--Hastings step after each resampling step in Section \ref{sec:MH}.
Finally, we present our promising numerical results in Section \ref{sec:numerics}. 
The appendices contain  proofs of auxiliray results and the pseudo-code of the algorithms.
\section{Preliminaries}\label{sec:prelim}

Let $\mathscr{M}(\mathbb{R}^{d})$ denote the set of finite \emph{positive} Borel measures on $\mathbb{R}^{d}$,
$\mathscr{P}(\mathbb{R}^{d})$  the set of  probability measure, and $\mathscr{P}_{2}(\mathbb{R}^{d})$ its subset of measures with finite second moments. 
For convenience, we just write $\int$ instead of $\int_{\R^d}$ in the following.

The \emph{Kullback-Leibler} (KL) \emph{divergence} is defined, 
for $\mu,\nu \in \mathscr{M}(\mathbb{R}^{d})$, as
$$
\kl(\mu\,\|\,\nu)=\int \log \left(\frac{\d \mu}{\d \nu} \right) \, \d\mu - \mu(\R^d) + \nu(\R^d)
$$
if $\mu$ is absolutely continuous with respect to $\nu$,
so that the Radon-Nikodym derivative $\frac{\d \mu}{\d \nu}$ exits and $+\infty$ otherwise.
We have $\kl(\mu,\nu)\geq 0$ with equality if and only if $\mu=\nu$.
We will need the following proposition which states 
the unit ball in the KL divergence is sequentially compact with respect to weak convergence.  

\begin{proposition}\label{prop:KL_balls_compact}
Let $r>0$ and $\pi\in\mathscr{M}(\R^d)$. Then, the set $\{\rho\in\mathscr{M}(\R^d):\kl(\rho\|\pi)\leq r\}$ is weakly sequentially compact.
\end{proposition}

 The proof is a standard consequence of  Prokhorov's theorem and the weak lower semicontinuity of KL. For convenience we give it in Appendix. 

Let $C_b(\R^d)$ denote the space of continuous bounded functions on $\R^d$.
A sequence $(\mu_k)_k$ in $\mathscr{M}(\mathbb{R}^{d})$
\emph{converges weakly} to $\mu \in \mathscr{M}(\mathbb{R}^{d})$, 
written $\mu_n \weakly \mu$,
if $\int \varphi \, \d \mu_n \to \int \varphi \, \d \mu$ for all $\varphi \in C_b(\R^d)$. 

By $L^p(\mu)$, $1 \le p  < \infty$, we denote the Banach space of (equivalence classes)
of functions which $p$-the power is $\mu$-integrable with the corresponding norm. We write $L^p(\mu,\R^d)$ for the corresponding space of $p$-integrable vector fields.

In the following, we recall  the Wasserstein-2, the Fisher--Rao, and the Wasserstein--Fisher--Rao distances on these spaces, together with the gradient flows they induce.

\subsection{Wasserstein Metric}

The \emph{Wasserstein(-2) distance} between $\mu,\nu\in\mathscr{P}_{2}(\mathbb{R}^{d})$ is defined as
\begin{equation}\label{eq:W2}
W_{2}^{2}(\mu,\nu) 
\coloneqq \inf_{\gamma\in\Gamma(\mu,\nu)}
\int_{\mathbb{R}^{d}\times\mathbb{R}^{d}} |x-y|^{2}\, \d \gamma(x,y),
\end{equation}
where 
$
\Gamma(\mu,\nu) \coloneqq \{\gamma \in \mathscr{P}_{2}(\mathbb{R}^{d} \times \R^d: (P_1)_\# \gamma = \mu, \, (P_2)_\# \gamma = \nu\}
$ 
with projections $P_1(x,y) = x$, $P_2(x,y) = y$
and the push-forward measures 
$(P_i)_\# \gamma \coloneqq \gamma \circ P_i$, $i=1,2$.
equipped with $W_2$, the set $\mathscr{P}_2$ becomes a metric space.
Further, we have that $W_{2}(\mu_k,\mu) \to 0$ as $k \to \infty$ if and only if $\mu_k \weakly \mu$  and $\int |x|^2 \, \d\mu_k \to \int |x|^2 \, \d\mu$ as $k \to \infty$.

A weakly continuous curve $\mu_t: (0,1)\to \mathscr{P}_2(\mathbb R^d)$  is \emph{absolutely continuous} if and only if there exists a Borel measurable vector field $v:(0,1)\times \mathbb R^d \to \mathbb R^d$ such that $\|v_t\|_{L^2(\mu_t,\R^d)}\in L^1(0,1)$ and
$(\mu_t,v_t)$ fulfills the \emph{continuity equation}
\begin{equation}\label{eq:continuity}
\partial_{t}\mu_{t}+\div(\mu_{t}v_{t})=0
\end{equation}
in a distributional sense.
The \emph{Benamou--Brenier formula}~\cite{Benamou2000computational} recasts the Wasserstein metric as the minimal kinetic energy needed to connect two probability measures as 
\begin{equation}\label{eq:W2:BB}
W_{2}^{2}(\mu,\nu) = \inf_{(\mu_t,v_t)} \Biggl\{\int_{0}^{1}\int |v_{t}|^{2}\,\mathrm{d}\mu_{t} \, \d t \colon
\begin{aligned}
&\partial_{t}\mu_{t}+\div(\mu_{t}v_{t})=0, \\
&\mu_0=\mu,~\mu_1=\nu
\end{aligned}
\Biggr\}.
\end{equation}

The following proposition assiciates absolutely continuous curves with ODE solutions, see \cite[Proposition 8.1.8, Lemma 8.1.4]{Ambrosio2005Gradient} in a one-to-one fashion.

\begin{proposition}\label{prop:basic_w}
Let $v\colon[0,1]\times\mathbb{R}^{d}\to\mathbb{R}^{d}$ fulfill 
\begin{align}\label{lip_v}
\int_I \sup_{x \in \R^d} |v_t(x)|+\mathrm{Lip}(v_t) \, \d t<\infty.
\end{align}
Then, for every $x \in \R^d$, the ODE
\begin{align}\label{eq:flow_ode} 
\partial_t \phi_t = v_t(\phi_t), \quad \phi(0,x)=x
\end{align}
admits unique solution $\phi_t =\phi(t,x) :I\times\mathbb R^d\to \R^d$ on $[0,1]$ and the map $\phi_{t}\colon \mathbb{R}^{d}\to\mathbb{R}^{d}$, $t \in [0,1]$  is a $C^{1}$-diffeomorphisms and
its determinant of the Jacobian $\det\nabla\phi_{t}$ fulfills the Liouville formula 
\begin{equation} \label{eq:det}
\det\nabla\phi_{t}(x) = \exp\Bigl(\int_{0}^{t}\div v_{s}(\phi_{s}(x))\,\mathrm{d}s\Bigr).
\end{equation}
The curve 
\begin{equation} \label{eq:push}
\mu_t \coloneqq \phi(t,\cdot)_\sharp\mu_0
\end{equation}
is the unique weakly continuous solution of  the continuity equation \eqref{eq:continuity} with the above vector field.
\end{proposition}

\subsection{Fisher--Rao Metric}
The \emph{Fisher--Rao metric} for $\mu,\nu\in\mathscr{M}(\mathbb{R}^{d})$ is defined by
\begin{equation}\label{eq:FR}
\fr^{2}(\mu,\nu) \coloneqq 4\int \left(\sqrt{\frac{\d \mu}{\d (\mu+\nu)}}-\sqrt{\frac{\d\nu}{\d (\mu+\nu)}} \right)^{2}\,\d (\mu+\nu). 
\end{equation}
The Fisher--Rao curve $\mu_t: (0,1) \to \mathscr{M}(\mathbb R^d)$ is determined by the  reaction equation
\begin{equation}\label{eq:reaction}
\partial_{t}\mu_{t}=g_{t}\mu_{t}, 
\end{equation}
where the measurable function $g_{t}\colon\mathbb{R}^{d}\to\mathbb{R}$ is the pointwise rate of mass creation. Analogously to~\eqref{eq:W2:BB}, the metric~\eqref{eq:FR} admits the dynamical representation
\begin{equation}\label{eq:FR:reaction}
\fr^{2}(\mu,\nu) = \inf_{(\mu_t,g_t)}\Biggl\{\int_{0}^{1}\int g_{t}^{2}\,\mathrm{d}\mu_{t}\mathrm{d}t \colon
\begin{aligned}
&\partial_{t}\mu_{t}=g_{t}\mu_{t}, \\
&\mu_0=\mu,~\mu_1=\nu
\end{aligned}
\Biggr\}.
\end{equation}

\subsection{Wasserstein--Fisher--Rao Metric}
\label{subsec:WFR}

Combining the transport and reaction geometries yields a metric on  $\mathscr{M}(\mathbb R^d)$ that allows mass to be simultaneously transported and created or destroyed. The corresponding dynamics obeys the \emph{continuity equation with reaction}
\begin{equation}\label{eq:continuity:reaction}
\partial_{t}\mu_{t}+\div(\mu_{t}v_{t})=g_{t}\,\mu_{t}.
\end{equation}
in the sense of distributions. We assume that the coefficients $v$ and $g$ are Hellinger-Kantorovich regular, namely, they are Borel measurable and satisfy $\|v_{t}\|_{L^{2}(\mu_t,\mathbb{R}^{d})}\in L^{1}(0,1)$ and $\|g_{t}\|_{L^{2}(\mu_{t})}\in L^{1}(0,1)$.

The \emph{Wasserstein--Fisher--Rao} (WFR) \emph{metric}, also known as \emph{ Hellinger--Kantorovich metric}~\citep{Chizat2018Interpolating}, with parameter $\alpha>0$ is defined as the minimal total kinetic plus reactive action  required to interpolate between  $\mu,\nu\in\mathscr{M}(\mathbb{R}^{d})$:
\begin{equation}\label{eq:WFR}
\wfr_{\alpha}^{2}(\mu,\nu) =
\inf_{(\mu_t,v_t,g_t)}\Biggl\{\int_{0}^{1}\int
|v_{t}|^{2}+\frac{1}{\alpha}g_{t}^{2}\, \d\mu_{t}\d t \colon
\begin{aligned}
&\partial_{t}\mu_{t}+\div(\mu_{t}v_{t})=g_{t}\,\mu_{t}, \\
&\mu_0=\mu,\ \mu_1=\nu
\end{aligned}
\Biggr\}.
\end{equation}
Setting $g\equiv0$ recovers the Benamou-Brenier formula~\eqref{eq:W2:BB}, while setting $v\equiv0$ gives the Fisher--Rao representation~\eqref{eq:FR:reaction}; in this sense $\wfr_{\alpha}$ interpolates between the Wasserstein and Fisher--Rao metrics as $\alpha$ varies.
We note that changing the parameter $\alpha$ in the Wasserstein--Fisher--Rao metric corresponds to rescaling, see also \cite[p.~13]{liero2023fine}.
More precisely, it follows directly from the its definition  that
\begin{equation}\label{eq:scale}
\wfr_{\alpha_1}^2(\mu,\nu)=\frac{\alpha_2}{\alpha_1}\wfr_{\alpha_2}^2((s_{\alpha_1/\alpha_2})_\#\mu,(s_{\alpha_1/\alpha_2})_\#\nu),\quad s_\alpha(x)=\sqrt{\alpha}x.
\end{equation}
Further, the WFR metric metrizes weak convergence:

\begin{proposition}\cite[Thm 7.15]{liero2018optimal}\label{thm:wfr_met_weak_conv}
Let $\mu_k \in\mathscr{M}(\R^d)$. Then $\wfr_{\alpha} (\mu_k, \mu) \to 0 $
if and only if $\mu_k \weakly \mu$ as $k \to \infty$.
\end{proposition}

In \cite{liero2018optimal}, Liero, Mielke and Savare proved that the WFR metric is actually equivalent to an unbalanced optimal transport problem called logartihmic entropy transport. More precisely, we have the following theorem which was shown for $\alpha=4$ and generalizes for arbitrary $\alpha>0$ by the rescaling property \eqref{eq:scale}. 

\begin{theorem} \cite[Thm 7.20]{liero2018optimal}
Let $\mu,\nu\in\mathscr{M}(\R^d)$. Then the WFR metric coincides with the logartihmic entropy transport (LET) problem given by
\begin{equation}\label{eq:let_problem}
\wfr_\alpha^2(\mu,\nu)=\frac{4}{\alpha}\min_{\gamma\in\mathscr{M}(\R^d\times\R^d)}\{\kl(\gamma_1\|\mu)+\kl(\gamma_2\|\nu)+\int\ell_\alpha \, \d\gamma\},
\end{equation}
where $\gamma_1$ and $\gamma_2$ are the marginals of $\gamma$ and $\ell_\alpha$ is the cost function
$$
\ell_\alpha(x,y)\coloneqq
\begin{cases}
-\log(\cos^2(\frac{\sqrt{\alpha}}{2}\|x-y\|),&\text{if }\|x-y\|<\frac{\pi}{\sqrt{\alpha}},\\
+\infty,&\text{otherwise.}
\end{cases}
$$
\end{theorem}

The following proposition is a direct conclusion of~\cite[Propositions 2.3 and 2.4]{Bredies2022superposition} and was shown for bounded $g$ instead of 
also in \cite{Maniglia2007}.

\begin{proposition}\label{prop:solution:continuity:reaction}
Let $v\colon[0,1]\times\mathbb{R}^{d}\to\mathbb{R}^{d}$ satisfy \eqref{lip_v}, and let $g\colon[0,1]\times\mathbb{R}^{d}\to\mathbb{R}$ be measurable with
\begin{align}\label{eq:reaction:regularity}
\int_{0}^{1}\sup_{x\in\mathbb{R}^{d}}|g_{t}(x)|+\operatorname{Lip}(g_{t}) \d t 
&<\infty
\end{align}
as well as
\begin{align}\label{eq:reaction:regularity_add}
\int_{0}^{1} \int \sup_{x\in\mathbb{R}^{d}} |v_{t}| + |g_t| \, \d \mu_t \d t 
&<\infty.
\end{align}
Let $\phi\colon[0,1]\times\mathbb{R}^{d}\to\mathbb{R}^{d}$ be the solution of the ODE \eqref{eq:flow_ode}.
Then the curve
\begin{equation}\label{eq:nu}
\mu_{t} \coloneqq 
\phi(t, \cdot)_\# (w_{t}\mu_{0}) , \quad
w_{t}(x) \coloneqq 
\exp\left(\int_{0}^{t}g_{s} \left(\phi(s,x) \right)\d s\right) \quad \mu_{0}-a.e.
\end{equation}
is the unique weakly continuous solution of the continuity equation with reaction ~\eqref{eq:continuity:reaction}.
\end{proposition}

\section{Wasserstein--Fisher--Rao JKO Scheme}\label{section:WFR:JKO}
In this section, we consider the Wasserstein--Fisher--Rao
Jordan-Kinderlehrer-Otto scheme for minimizing functionals on $\mathscr{M}(\mathbb{R}^{d})$ and prove exponential convergence
for the KL functional $F = \kl(\cdot||\pi)$, where $\pi >0$.

\subsection{WFR JKO Scheme}
 Let $\mathcal{F}\colon\mathscr{M}(\mathbb{R}^{d})\to[0,+\infty]$ be a proper functional 
 which we intend to minimize by an implicit gradient descent algorithm in $\mathscr{M}(\mathbb{R}^{d})$ equipped with the WFR metric. The corresponding scheme in the Wasserstein space is known as Jordan-Kinderleher-Otto (JKO) scheme \cite{JKO1998} and we adopt the notation for current geomtry.
 Given $\mu_k \in \mathscr{M}(\mathbb{R}^{d})$, one step of the WFR JKO scheme is defined, whenever a minimizer exists, by
\begin{equation}\label{eq:wfr:jko}
\mu_{k+1} \in\argmin_{\mu\in\mathscr{M} (\mathbb{R}^{d})}\Bigl\{\frac{1}{2h}\wfr_{\alpha}^{2}(\mu,\mu_k)+\mathcal{F}(\mu)\Bigr\}.
\end{equation}
 In this paper, we are only interested in 
 $$\mathcal{F} = \kl(\cdot \,\|\, \pi).$$
Then the sequence in \eqref{eq:wrf_jko_basic} is well-defined. For convenience, we add the small proof.

\begin{lemma} \label{lem:exist}
Let $\pi,\mu_k\in  \mathscr{M}(\R^d)$ with $\kl(\mu_k\|\pi)<\infty$. Then, the minimum in \eqref{eq:wrf_jko_basic} is attained.
\end{lemma}

\begin{proof}
Let $\phi_h(\mu_k)$ be the infimum in \eqref{eq:wrf_jko_basic},
and let $\nu_n\in\mathscr{M}(\R^d)$ constitute a minimizing sequence such that 
$$
\frac{1}{2h}\wfr_{\alpha}^2(\nu_n,\mu_k)+\kl(\nu_n\|\pi)\leq \phi_h(\mu_k)+\frac{1}{n}.
$$
By Proposition~\ref{prop:KL_balls_compact}, we know, up to subsequences, that $\nu_n$ converges weakly to some $\nu\in\mathscr{M}(\R^d)$. Further, we know by Proposition~\ref{thm:wfr_met_weak_conv} that $\wfr_{\alpha}^2(\cdot,\mu_k)$ is continuous wrt weak convergence. Further, $\kl(\cdot\|\pi)$ is lower semicontinuous wrt weak convergence.
Thus
$$
\frac{1}{2h}\wfr_{\alpha}^2(\nu,\mu_k)+\kl(\nu\|\pi)=\phi_h(\mu_k),
$$
which means that $\nu\in\argmin_{\mu}\frac{1}{2h}\wfr_{\alpha}^2(\mu,\mu_k)+\kl(\mu\|\pi)$.
\end{proof}

In  the rest of this paper, let 
$\pi > 0$ be integrable on $\R^d$. With slight abuse of notation we identify the function $\pi$ with the corresponding measure $\pi \, \d x$. 
For such $\pi$, the WFR JKO scheme becomes
\begin{equation}\label{eq:wrf_jko_basic}
\rho_{k+1}\in\argmin_{\rho\in\mathscr{M}(\R^d)} J(\rho), \quad J(\rho) \coloneqq \frac{1}{2h}\wfr_{\alpha}^2(\rho,\rho_k)+\kl(\rho\|\pi).
\end{equation}
Note that by  this definition $\kl(\rho_{k+1}\|\pi)$ is finite for $k\geq 0$ such that $\rho_{k+1}$ is absolutely continuous wrt $\pi$ and thus also wrt the Lebesgue measure. Therefore, we rely on densities $\rho$.

\begin{remark}[Gradient Flows]
Several papers study gradient flows of the KL divergence or other functionals with respect to the WFR metric, see, e.g., \cite{crucinio2025sequential,gallouet2017jko,lu2023birth,mielke2025hellinger,zhu2024kernel}. Formally, such gradient flows can be defined as solutions of the reaction diffusion equation
\begin{equation}\label{eq:wrf_gradient_flows}
\partial_t\mu_t-\div\left(\mu_t\nabla F'(\mu_t)\right)=-\alpha\mu_tF'(\mu_t),
\end{equation}
where $F'(\mu)=\frac{\delta \mathrm{KL}(\cdot,\pi)}{\delta\mu}=\log\left(\frac{\mu}{\pi}\right)+1$ is the first variation of the KL divergence.
The sequence $(\rho^h_k)_k$ generated by the WFR JKO scheme \eqref{eq:wrf_jko_basic} for some step size $h$ can be viewed as a time-discretized version of such curves. More precisely, the authors of \cite{fleissner2021minimizing, Laschos04032026} study limits of the curves $\mu_t^h\colon[0,\infty)\to\mathscr M(\R^d)$ defined as
$
\mu_{t}^h=\rho_k^h
$
for $t\in[kh,(k+1)h)$. In particular, it was shown by Fleißner \cite[Example 4.12]{fleissner2021minimizing} under mild assumptions that $\mu_t^h$ converges weakly to a weak solution of \eqref{eq:wrf_gradient_flows} as $h\to 0$. In contrast, we merely study the convergence of $(\rho_k^h)_k$ for fixed $h$ as $k\to\infty$. 
\end{remark}

\subsection{Convergence of WFR JKO for $k \to \infty$}\label{sec:convergence}

In this subsection, we prove exponential convergence of the sequence $\rho_k$ generated by the WFR JKO scheme \eqref{eq:wrf_jko_basic} towards $\pi$ for $k\to\infty$. We use the notation
$$
I(\rho)=\int\rho \log^2 \left(\frac{\rho}{\pi} \right) \d x.
$$
Then, the following descent lemma holds true.
The main idea of the proof is to bound the WFR metric by the FR metric and to insert the optimality condition of $\rho_{k+1}$ along the geometric interpolation between $\rho_{k+1}$ and $\pi$.

\begin{lemma}\label{lem:sufficient_decrease}
Let $(\rho_k)_k$ be generated by \eqref{eq:wrf_jko_basic}. Then, we have for any $k\geq0$ that
$$
\alpha I(\rho_{k+1})\leq \frac{1}{h^2}\wfr^2(\rho_{k+1},\rho_k).
$$
\end{lemma}

\begin{proof}
For $s \in [0,1]$, let $\mu_s \coloneqq \rho_{k+1}^{1-s}\pi^s$. 
Then $\mu_{st}$, $t \in [0,1]$ fulfills $(\mu_{st})|_{t=0} = \rho_{k+1}$ and
$(\mu_{st})|_{t=1} = \mu_s$ and
$$
\partial_t \mu_{st} = s \log \left(\frac{\pi}{\rho_{k+1}}\right) \,  \mu_{st} = g_t \, \mu_{st}, \quad g_t \coloneqq s\log (a), \quad a \coloneqq \frac{\pi}{\rho_{k+1}}.
$$
Then we conclude that the $\wfr$ distance between $\rho_{k+1}$ and $\mu_s$ is bounded by
\begin{align*}
\wfr_\alpha^2(\rho_{k+1},\mu_s)
&\leq\int_0^1\int \frac{\mu_{st}}{\alpha}g_t^2\d x\d t
=\frac{s^2}{\alpha}\int_0^1\int \pi^{st}\rho_{k+1}^{1-st}\log^2(a)  \, \d x\d t\\
&= \frac{s^2}{\alpha} \int \rho_{k+1}\log^2(a) \int_0^1 a^{st} \, \d t \d x
=\frac{s}{\alpha}\int \rho_{k+1}\log(a)\left( a^{s}-1\right) \, \d x .
\end{align*}
Next, we consider the function
$H(s)\coloneqq\kl(\rho_{k+1}\|\pi)-\kl(\mu_s\|\pi)$. We observe that $H(0)=0$ such that $H(s)=\int_0^sH'(t) \d t$. 
We compute $H'(t)$ by
\begin{align*}
H'(t)&=-\frac{\d}{\d t} \kl(\mu_t\|\pi)
=\frac{\d}{\d t} \int \rho_{k+1} \left(
(1-t)\log(a) \, a^t+ a^t \right)  \d x \\
&=(1-t)\int\rho_{k+1} \log^2 (a) \, a^t \, \d x\geq 0.
\end{align*}
In particular, since the integrant is positive, we obtain
\begin{align*}
H(s)&=\int_0^sH'(t)\d t
\geq (1-s) \int \rho_{k+1} \log^2 (a) \int_0^1 a^t \, \d t\d x\\
&=(1-s)\int \rho_{k+1}\log(a) (a^s-1) \, \d x\\
&\geq \frac{(1-s)\alpha}{s} \wfr_\alpha^2(\rho_{k+1},\mu_s).
\end{align*}
By the optimality of $\rho_{k+1}$ and applying the triangle inequality in the first summand, we get 
\begin{align*}
0&\leq \frac{1}{2h}\wfr_\alpha^2(\mu_s,\rho_k)-\frac{1}{2h}\wfr_\alpha^2(\rho_{k+1},\rho_k)-H(s)\\
&\leq \frac{1}{2h}\wfr_\alpha^2(\rho_{k+1},\mu_s)+\frac{1}{h}\wfr_\alpha(\rho_{k+1},\mu_s)  \wfr_\alpha(\rho_{k+1},\rho_k)-H(s)\\
&\leq \left(\frac{s}{2h(1-s)\alpha}-1 \right) H(s)+\frac{1}{h}\sqrt{\frac{s}{(1-s)\alpha} H(s)} \wfr_\alpha(\rho_{k+1},\rho_k).
\end{align*}
Division by $\sqrt{H(s)}$ gives
$$
\left(1-\frac{s}{2h(1-s)\alpha} \right) \sqrt{H(s)}\leq \frac{1}{h}\sqrt{\frac{s}{(1-s)\alpha}} \wfr_\alpha(\rho_{k+1},\rho_k).
$$
Since for $0 \le s < 2h \alpha /(1+2h \alpha)$ both sides are nonnegative, taking  squares  yields
$$
\left(1-\frac{s}{2h(1-s)\alpha} \right)^2 H(s)\leq \frac{1}{h^2}\frac{s}{(1-s)\alpha} \wfr_\alpha^2(\rho_{k+1},\rho_k).
$$
Since for $s=0$ both sides are zero, we have that the derivative of the left side must be smaller or equal than the derivative on the right side, i.e.,
$$
H'(0)\leq \frac{1}{\alpha h^2}\wfr_\alpha^2(\rho_{k+1},\rho_k)
$$
Noting that $H'(0)=I(\rho_{k+1})$ yields the assertion.
\end{proof}

The next lemma shows that $\rho_{k}/\pi$  remains bounded from below.

\begin{lemma}\label{lem:ratio_bound}
Let $(\rho_k)_k$ be generated by \eqref{eq:wrf_jko_basic}.
Then, it holds for all $k \in \mathbb N$ that $\rho_{k} \geq \exp(-\frac{2}{\alpha h})\pi$ a.e..
\end{lemma}

\begin{proof}
Let $\gamma$ be the optimal plan in the LET problem \eqref{eq:let_problem} between the marginals $\rho_{k+1}$ and $\rho_k$. 
We note that its marginals $\gamma_1$ and $\gamma_2$ are absoulutely continuous, since $\kl(\gamma_1\|\rho_{k+1})$ is finite. We define the energy
$$
\mathcal E(\rho)=\frac{2}{h\alpha}\left( \kl(\gamma_1\|\rho)+\kl(\gamma_2\|\rho_k)+\int\ell_\alpha\d \gamma \right) +\kl(\rho\|\pi).
$$
By definition, we have $\rho_{k+1}\in\argmin_{\rho\in\mathscr{M}(\R^d)} J(\rho)$. Further, by \eqref{eq:let_problem}, it holds 
$J(\rho_{k+1})= \mathcal E(\rho_{k+1})$ and 
$J(\rho)\leq \mathcal E(\rho)$ for all $\rho\in\mathscr{M}(\R^d)$.

Now we define the set $A\coloneqq\{x\in\R^d:\rho_{k+1}(x)<\exp(-\frac{2}{\alpha h})\pi(x)\}$ and show that $\pi(A)=0$. Assume by contradiction that $\pi(A)>0$ and define $\rho^*$ by 
$$
\rho^* \coloneqq\rho_{k+1}+\frac12 w,\qquad w(x)\coloneqq
\begin{cases}
\exp(-\frac{2}{h\alpha})\pi(x)-\rho_{k+1}(x),&\text{if }x\in A,\\
0,&\text{if }x\not\in A.
\end{cases}
$$
Note that by definition of $A$, the function $w$ is non-negative and fulfills 
$$\int w \, \d x\leq \exp(-\frac{2}{h\alpha})\pi(\R^d)<\infty$$ 
such that 
$\rho^* \in\mathscr{M}(\R^d)$. Next, we will bound 
\begin{equation}\label{eq:energy_difference}
\mathcal E(\rho^*)-\mathcal E(\rho_{k+1})=\frac{2}{h\alpha}(\kl(\gamma_1\|\rho^*)-\kl(\gamma_1\|\rho_{k+1}))+\kl(\rho^*\|\pi)-\kl(\rho_{k+1}\|\pi).
\end{equation}
To bound the first difference, we compute
\begin{align}
\kl(\gamma_1\|\rho^*)-\kl(\gamma_1\|\rho_{k+1})
&=\int \gamma_1 \log\left(\frac{\rho_{k+1}}{\rho^*} \right) + \rho^*-\rho_{k+1} \, \d x\\
&=\int \gamma_1 \log \left(\frac{\rho_{k+1}}{\rho_{k+1}+\frac12 w} \right) + \frac12 w \, \d x
\leq \frac12 \int w\,\d x, \label{eq:bound1}
\end{align}
where the last step follows from the fact that $w\geq 0$.
To bound the second difference, we use the nonnegative, convex function $\varphi(r) \coloneqq r \log r -r +1$ for $r \ge 0$. Then we have $\varphi(a)-\varphi(b)\leq \varphi'(a)(a-b)$ and $\varphi'(r)=\log(r)$ such that
\begin{align*}
\kl(\rho^*\|\pi)-\kl(\rho_{k+1}\|\pi)
&=\int \pi \left(\varphi\left(\frac{\rho^*}{\pi} \right)
-
\varphi\left (\frac{\rho_{k+1}}{\pi} \right) \right) \, \d x\\
&\leq 
\int \pi \left(\varphi' \left(\frac{\rho^*}{\pi} \right) \frac{\rho^*-\rho_{k+1}}{\pi} \right) \, \d x\\
&=
\frac12 \int \log\left(\frac{\rho_{k+1}+\frac12 w}{\pi} \right) w \, \d x.
\end{align*}
On $A$, we have  $\rho_{k+1}=\exp(-\frac{2}{h\alpha})\pi-w$ and consequently 
$$
\log\left(\frac{\rho_{k+1}(x)+\frac12 w(x)}{\pi(x)} \right)w(x)
=
\log \left(\exp\left(-\frac{2}{h\alpha} \right)-\frac{w(x)}{2\pi(x)} \right)w(x), \quad x \in A.$$ 
Outside of $A$, we have that $w(x)=0$ such that the same equality holds. Together, we obtain 
\begin{align*}
\frac12 \int \log\left(\frac{\rho_{k+1}+\frac12 w}{\pi} \right) w \, \d x
&= 
\frac12\int \log\left(\exp(-\frac{2}{h\alpha})-\frac{w}{2\pi} \right)w \, \d x\\
&=
\frac12\int -\frac{2}{h\alpha}w + 
\log\left(1-\frac{\exp(\frac{2}{h\alpha})w}{2\pi} \right) w \, \d x\\
&=
-\frac{1}{h\alpha}\int w\,\d x+\frac12\int\log\left(1-\frac{\exp(\frac{2}{h\alpha}) w}{2\pi} \right) w \, \d x
< -\frac{1}{h\alpha}\int w\,\d x,
\end{align*}
where the last inequality follows from our assumption $\pi(A)>0$ (implying $|A|>0$) and the fact that $w>0$ on $A$ and the second integrand is strictly negative, whenever $w>0$. In summary, we obtain
$$
\kl(\rho^*\|\pi)-\kl(\rho_{k+1}\|\pi)<-\frac{1}{h\alpha}\int w\,\d x
$$
Inserting this together with \eqref{eq:bound1} in \eqref{eq:energy_difference}, we obtain that $\mathcal E(\rho^*)-\mathcal E(\rho_{k+1})<0$. In particular, we have
$$
J(\rho^*)\leq \mathcal E(\rho^*)<\mathcal E(\rho_{k+1})= J(\rho_{k+1}),
$$
which contradicts the assumption that $\rho_{k+1}$ minimizes $J$. Thus, we conclude $\pi(A)=0$.
\end{proof}

Using the two lemmas, we can  prove our final convergence result. 

\begin{theorem}\label{thm:main}
Let $\pi \in \mathscr{M}(\R^d)$ with $\pi >0$ and $\rho_0 \in \mathscr{M}(\R^d)$. 
$(\rho_k)_k$ be generated by \eqref{eq:wrf_jko_basic}. 
Then, it holds
$$
\kl(\rho_{k+1}\|\pi)\leq \frac{1}{1+c}\kl(\rho_k\|\pi),
$$
such that
$$\kl(\rho_{k+1}\|\pi)\leq (1+c)^{-k}\kl(\rho_0\|\pi)
$$
for 
$$
c=\frac{\beta}{\exp(\beta)-\beta-1},\quad\text{where}\quad\beta=\frac{2}{\alpha h}.
$$
\end{theorem}

\begin{proof}
The function $\psi(r)\coloneqq\frac{\exp(r)-1-r}{r^2}$ is strictly increasing on $\R$
and is related to $\varphi(r) \coloneqq r\log(r) - r + 1$  by
$$
\varphi(r) = \psi\left(-\log(r) \right)r\log^2(r), \quad r \ge 0.
$$
Then we have 
\begin{align*}
\kl(\rho_{k+1}\|\pi)
=
\int \pi \varphi\left(\frac{\rho_{k+1}}{\pi} \right) \, \d x
=
\int \rho_{k+1} \log^2 \left(\frac{\rho_{k+1}}{\pi} \right)  \psi\left(-\log \left(\frac{\rho_{k+1}}{\pi} \right) \right) \, \d x.
\end{align*}
By Lemma~\ref{lem:ratio_bound} we know  that $\frac{\rho_{k+1}}{\pi}\geq \exp(-\frac{2}{\alpha h})$, i.e.,
$-\log \left(\frac{\rho_{k+1}}{\pi} \right) \le \frac{2}{\alpha h}$
and  since $\psi$ is increasing, this implies
\begin{equation}\label{eq:kl_bound}
\kl(\rho_{k+1}\|\pi)
\leq 
\int \rho_{k+1} \log^2\left(\frac{\rho_{k+1}}{\pi}\right)
\psi\left(\frac{2}{\alpha h} \right) \, \d x
=
\psi\left(\frac{2}{\alpha h} \right) I(\rho_{k+1} ).
\end{equation}
By definition $\rho_{k+1}$ minimizes $J$ in \eqref{eq:wrf_jko_basic} such that $J(\rho_{k+1})\leq J(\rho_k)$. Thus, using Lemma~\ref{lem:sufficient_decrease} and \eqref{eq:kl_bound},  we obtain the estimate
\begin{align}
\kl(\rho_k\|\pi)&\geq \frac{1}{2h}\wfr_\alpha^2(\rho_{k+1},\rho_k)+\kl(\rho_{k+1}\|\pi)\\
&\geq\frac{\alpha h}{2}I(\rho_{k+1})+\kl(\rho_{k+1}\|\pi)\\
&\geq \left(\frac{\alpha h}{2 \psi(\frac{2}{\alpha h})} + 1\right)\kl(\rho_{k+1}\|\pi).
\end{align}
This yields the claim with 
$$
c = \frac{\alpha h}{2 \psi(\frac{2}{\alpha h})}=\frac{\beta}{\exp(\beta)-\beta-1} >0.
$$
This finishes the proof.
\end{proof}

\section{Normalizing Flow via WFR JKO Scheme} \label{sec_norm_glows}

Next, we derive a computable loss function for training the velocity and reaction function in a single WFR--JKO step. Afterwards, we summarize the algorithm in order to simulate the whole scheme.

\subsection{Loss Function}
We consider the admissible class
\begin{equation*}
\mathcal{A} \coloneqq 
\left\{(v,g)\colon 
\begin{aligned}
&v\colon[0,1]\times\mathbb{R}^{d}\to\mathbb{R}^{d},\;
g\colon[0,1]\times\mathbb{R}^{d}\to\mathbb{R}~\text{ measurable }\\
&\text{ and fulfill } \eqref{lip_v},
\eqref{eq:reaction:regularity} \text{ and }
\eqref{eq:reaction:regularity_add}
\end{aligned}
\right\}.
\end{equation*}
By Proposition~\ref{prop:solution:continuity:reaction}, we know that for $(v,g)\in\mathcal{A}$, the weakly continuous curve $\mu\colon[0,1]\to\mathscr{M}(\mathbb{R}^{d})$ satisfying the continuity equation with reaction~\eqref{eq:continuity:reaction} is uniquely determined. Therefore, restricting the JKO step~\eqref{eq:wrf_jko_basic} to $(v,g) \in \mathcal A$ leads to the problem
\begin{equation}\label{eq:obj:WFR:vg}
\argmin_{(v_t,g_t)\in\mathcal{A}} \mathcal L_{k}(v,g), \quad 
\mathcal L_{k}(v,g) \coloneqq \frac{1}{2h}\int_{0}^{1}\int |v_{t}|^{2}+\frac{1}{\alpha}g_{t}^{2} \,\d\mu_{t}\d t +\kl(\mu_{1} \,\|\, \pi),
\end{equation}
where $\mu_{t} = \mu_t^{v_t,g_t}$ is the unique solution of the   the continuity equation with reaction~\eqref{eq:continuity:reaction} starting in $\mu_0 = \rho_k$.

In this section, we provide a computable formulation for the objective~\eqref{eq:obj:WFR:vg} of a WFR JKO step. 
Note that Wasserstein--Fisher--Rao schemes for neural sampling were considered, e.g. in~\citep{chemseddine2025neural}. The following proposition reformulates the action term in~\eqref{eq:obj:WFR:vg}. The proof follows standard lines in the normalizing flow literature, see, e.g.
\cite{Xu2023Normalizing,HHS2022}.

\begin{proposition}\label{proposition:joint:obj:W}
Let $(v,g) \in \mathcal A$. Further, let $\phi(t,x)$ be the solution of the ODE \eqref{eq:flow_ode} with vector field $v_t$ and 
$\mu_t \coloneqq \phi(t, \cdot)_\# (w_{t}\rho_k)$ with weight $w_t$ from \eqref{eq:nu}. Let $\pi(x)  = \exp(-V(x))$ with $Z \coloneqq \int \exp{-V(x)} \, \d x < \infty$.
Then it holds 
\begin{align}
\mathcal{L}_{k}(v,g) 
&= \frac{1}{2h}\int_{0}^{1}\mathbb{E}\Bigl[ \rho_{k}(\mathbb{R}^{d}) w_{t}(X_{k})\left(|v_{t}\circ\phi_{t}(X_{k})|^{2}
+\frac{1}{\alpha} g_{t}^2\circ\phi_{t}(X_{k}) \right)\Bigr]\,\d t  \\
&\quad +\mathbb{E}\Bigl[\rho_{k}(\R^d) \, w_{1}(X_{k}) \Bigl(\log\rho_{k}(X_{k})+\log w_{1}(X_{k}) +V\circ\phi_{1}(X_{k})-1\Bigr) \Bigr] \nonumber \\
&\quad -\int_{0}^{1}\mathbb{E}\Big[\rho_{k}(\R^d) \, w_{1}(X_{k}) (\div v_{t}) \circ\phi_{t}(X_{k})\Bigr]\,\mathrm{d}t + \pi(\R^d),
\end{align}
where the expectation is taken with respect to $X_{k}\sim\rho_{k}(\mathbb{R}^{d})^{-1}\rho_{k}$
\end{proposition}

\begin{proof}
First, we have, for any $t\in[0,1]$, by definition of the push-forward operator that
\begin{equation}\label{eq:helper}
\int_{\mathbb{R}^{d}}|v_{t}|^{2}+\frac{1}{\alpha}g_{t}^{2} \,\mathrm{d}\mu_{t} 
= 
\mathbb{E} \Bigl[\rho_{k}(\mathbb{R}^{d}) w_{t}(X_{k}) \left(
|v_{t}\circ\phi_{t}(X_{k})|^{2}+\frac{1}{\alpha}(g_{t}\circ\phi_{t}(X_{k}))^{2}\right) \Bigr].
\end{equation}

Second, we obtain by definition of the KL divergence that
\begin{equation}\label{eq:proposition:kl:WFR:1}
\kl(\mu_{1} \,\|\, \pi) 
= \underbrace{\int_{\mathbb{R}^{d}}\log\mu_{1}\,\mathrm{d}\mu_{1}}_{A_{1}}+\underbrace{\int_{\mathbb{R}^{d}}V\,\mathrm{d}\mu_{1}}_{A_{2}}-\underbrace{\mu_{1}(\mathbb{R}^{d})}_{A_{3}}+\pi(\mathbb{R}^{d}).
\end{equation}
First, we consider the summand $A_{1}$. For the density of the push-forward measure, we have by the change of variable formula and \eqref{eq:det} that
\begin{align}
\mu_{1}\circ\phi_{1} = \frac{w_1 \rho_k}{\det \nabla \phi_1} = w_1 \rho_k \exp\left( -\int_0^1 (\div v_t)\circ \phi_t \, \d t\right),
\end{align}
so that 
\begin{align}\label{eq:proposition:kl:WFR:2}
\log (\mu_{1}\circ\phi_{1})
&= \log\rho_{k}+\log w_{1} -\int_{0}^{1} (\div v_t)\circ \phi_t \,\mathrm{d}t \end{align} 
and
\begin{align}
\int \log\mu_{1}\,\mathrm{d}\mu_{1}
&= \int \log(\mu_{1}\circ\phi_{1}) w_{1}\,\mathrm{d}\rho_{k} \nonumber \\
&= \int\Bigl(\log\rho_{k}+\log w_{1}\Bigr)w_{1}\,\mathrm{d}\rho_{k} 
 -
 \int_{0}^{1}\Bigl(\int (\div v_t)\circ \phi_t w_{1}\,\mathrm{d}\rho_{k}\Bigr)\,\mathrm{d}t \nonumber \\
&= \mathbb{E}\Bigl[\Bigl(\log\rho_{k}(X_{k})+\log w_{1}(X_{k})\Bigr)w_{1}(X_{k}) \rho_k(\R^d)\Bigr] \nonumber \\
&\quad -\int_{0}^{1}\mathbb{E}\Big[(\div v_t)\circ \phi_t(X_{k})w_{1}(X_{k})  \rho_k(\R^d)\Bigr]\,\mathrm{d}t. \label{eq:proposition:kl:WFR:4}
\end{align}
For the summand $A_{2}$,  it follows by definition of the push-forward operator that 
\begin{equation}\label{eq:proposition:kl:WFR:6}
\int V \,\mathrm{d}\mu_{1} = \int  w_{1} (V \circ \phi_{1}) \,\mathrm{d}\rho_{k} = \mathbb{E}\bigl[ \rho_{k}(\R^d) \, w_{1}(X_{k})V(\phi_{1}(X_{k}))\bigr], 
\end{equation}
and the summand $A_{3}$ similarly
\begin{equation}\label{eq:proposition:kl:WFR:7}
\mu_{1}(\mathbb{R}^{d}) 
=\mathbb{E}[\rho_{k}(\R^d) \, w_{1}(X_{k})].
\end{equation} 
In summary, this gives
\begin{align*}
\kl(\nu_{1} \,\|\, \pi) 
&= \mathbb{E}\Bigl[ \rho_{k}(\R^d) \, w_{1}(X_{k}) \Bigl(\log\rho_{k}(X_{k})+\log w_{1}(X_{k})-\log\rho_{k}(\mathbb{R}^{d})+V(\phi_{1}(X_{k}))-1\Bigr)\Bigr] \nonumber \\
&\quad -\int_{0}^{1}\mathbb{E}\Big[(\div v_t)\circ \phi_t(X_{k})w_{1}(X_{k}) \rho_k(\R^d)\Bigr]\,\mathrm{d}t
\end{align*}
and adding \eqref{eq:helper} yields the assertion
\end{proof}

\begin{remark}
  Let $\xi\sim\mathrm{Rad}^d$ be Rademacher distributed, meaning that its components are iid
  $\xi_i \sim \text{Unif}\{-1,1\}$, and let $\xi$ be independent of $X_{k}$. Then Hutchinson's stochastic trace estimator \cite{Hutchinson1989} yields 
\begin{equation}\label{eq:proposition:kl:WFR:5}
\mathbb{E}\Bigl[\rho_{k}(\R^d) \,w_{1}(X_{k})(\div v_t)\circ \phi_t(X_{k})\Bigr] = \mathbb{E}\bigl[\rho_{k}(\R^d) \,w_{1}(X_{k})\xi^{\top}\nabla v_{t}\circ\phi_{t}(X_{k})\xi\bigr],
\end{equation}
where the expectation on the right-hand side is taken with respect to $(X_{k},\xi)\sim(\rho_{k}(\mathbb{R}^{d})^{-1}\rho_{k})\otimes\mathrm{Rad}$.    Alternatively, the Rademacher distribution could be replaced by the standard Gaussian one, but
$$
\mathbb E_{\text{Gauss}}[(\xi^{\top} A \xi- \text{trace}(A))^2] \ge \mathbb E_{\text{Rad}^d}[(\xi^{\top} A \xi- \text{trace}(A))^2], 
$$
see, e.g., \cite{MT2020}. 
\end{remark}

By Propositions~\ref{proposition:joint:obj:W}, we suggest the 
following objective function to learn neural networks for $(v,g)$:
\begin{align}
\mathcal{L}^{\wfr}_{k}(\theta) 
&= \frac{1}{2h}\int_{0}^{1}\mathbb{E}\Bigl[ \rho_{k}(\mathbb{R}^{d}) w_{t}(X_{k})\left(|v_{t}^\theta \circ\phi_{t}(X_{k})|^{2}
+\frac{1}{\alpha} (g_{t}^\theta \circ\phi_{t}(X_{k}))^2 \right)\Bigr]\,\d t  \\
&\quad +\mathbb{E}\Bigl[\rho_{k}(\R^d) \, w_{1}^\theta(X_{k})\Bigl(\log\rho_{k}(X_{k})+\log w_{1}^\theta(X_{k})+V\circ\phi_{1}(X_{k})-1\Bigr)\Bigr] \nonumber \\
&\quad -\int_{0}^{1}\mathbb{E}\Big[\rho_k(\R^d) w_{1}^\theta (X_{k}) (\div v_t)\circ \phi_t(X_{k})\Bigr]\,\mathrm{d}t. \label{eq:obj:WFR}
\end{align}
Likewise the final summand can be replaced by the Hutchinson estimator \eqref{eq:proposition:kl:WFR:5}.

 Given samples $X_{k}^{1},\ldots,X_{k}^{N}\in\mathbb{R}^{d}$ from the distribution $\rho_{k}(\mathbb{R}^{d})^{-1}\rho_{k}$ and independent samples $\xi_{1},\ldots,\xi_{N}\sim^{\mathrm{i.i.d.}}\mathrm{Rad}^{d}$, we approximate the population risk~\eqref{eq:obj:WFR} through an empirical risk 
\begin{align}
\widehat{\mathcal{L}}^{\wfr}_{k}(\theta) 
&= \frac{1}{2hN}\sum_{i=1}^{N}\int_{0}^{1}w_{t}^\theta(X_{k}^{i})\Bigl\{|v_{t}^\theta\circ\phi_{t}(X_{k}^{i})|^{2}+\frac{1}{\alpha}(g_{t}^\theta\circ\phi_{t}(X_{k}^{i}))^{2}\Bigr\}\,\mathrm{d}t \nonumber \\
&\quad +\frac{1}{N}\sum_{i=1}^{N}w_{1}(X_{k}^{i})\Bigl\{\log\rho_{k}(X_{k}^{i})+\log w_{1}^\theta(X_{k}^{i})+V\circ\phi_{1}(X_{k}^{i})-1\Bigr\} \nonumber \\
&\quad -\frac{1}{N}\sum_{i=1}^{N}\int_{0}^{1}
w_{1}^{\theta}(X_{k}^{i})\xi_{i}^{\top}\nabla v_{t}^\theta\circ\phi_{t}(X_{k}^{i})\xi_{i}\,\mathrm{d}t, \label{eq:obj:WFR:empirical}
\end{align}

\subsection{Algorithms}\label{sec:algorithms}

In order to simulate the WFR--JKO scheme, we start with a tractable latent measure $\rho_0$ which is usually a isometric Gaussian distribution $\mathcal N(0,\sigma_0^2 I)$, where $\sigma_0$ is the initial covariance which we consider as a hyperparameter. We draw initial samples $X_0^1,...,X_0^N$ from $\rho_0=\rho_0(\R^d)^{-1}\rho_0$. Then, we iteratively generate samples $X_k^1,...,X_k^N$ from $\rho_k(\R^d)^{-1}\rho_k$ and compute the corresponding densities $\rho_k(X_k^i)$ as well as the total mass $\rho_k(\R^d)$ by performing the following steps for $k=0,...,K$:
\begin{enumerate}
    \item We learn the velocity $v^k$ and reaction rate $g^k$ from the WRF--JKO scheme \eqref{eq:obj:WFR:vg} by parameterizing $v^k$ and $g^k$ as neural networks with parameters $\theta$ and minimizing the loss function $\widehat{\mathcal{L}}^{\wfr}_{k}(\theta)$ from \eqref{eq:obj:WFR:empirical}. We include the detailed computation of this loss function as pseudo-code in Algorithm \ref{alg:empirical:risk} in the appendix.
    \item Once the velocity and reaction are trained, we can draw weighted samples $Y_{k+1}^1,...,Y_{k+1}^N$ with weights $w_{k+1}^1,...,w_{k+1}^N$ from $\rho_{k+1}(\R^d)^{-1}\rho_{k+1}$ and the total mass $\rho_{k+1}(\R^d)$ by solving the flow ODE $\phi$ from \eqref{eq:flow_ode} and applying \eqref{eq:nu}. We outline this computation as pseudo-code in Algorithm \ref{alg:run:jko:step}.
    \item In order to obtain an ensemble of unweighted samples from $\rho_{k+1}(\R^d)^{-1}\rho_{k+1}$, we perform a resampling step as proposed in \cite{DDJ2006}. More precisely, we normalize the weights $w_{k+1}^1, ..., w_{k+1}^N$ and draw samples from the empirical distribution $\sum_{n=1}^Nw_{k+1}^n \delta_{Y_{k+1}^n}$. This step is outlined as pseudo-code in Algorithm \ref{alg:sample}.
\end{enumerate}
Repeating this procedure for $k=0,...,K$ provides us with samples $X_K^1,...,X_K^N$ from $\rho_k(\R^d)^{-1}\rho_k$ and the total mass $\rho_k(\R^d)$.

\section{Metropolis--Hastings Steps}\label{sec:MH}

\par The resampling step from of the section unavoidably causes particle degeneracy: low-weight particles are discarded and high-weight particles are duplicated, so the support of the ensemble collapses. To rejuvenate the ensemble while preserving its target, we apply one step of the Metropolis--Hastings (MH) algorithm, see, e.g., \cite{robertsrosenthalMCMC}, after each resampling step. As a stationary distribution we choose the distribution $\rho_k(\R^d)^{-1}\rho_k$ such that the overall distribution of the samples is not affected by the MH steps. In the following, we first revisit the MH algorithm and then focus on its application in the WFR--JKO scheme.

\subsection{Metropolis--Hastings Algorithm}

Let $\mathcal Q\colon\R^d\times\mathcal B(\R^d)\to[0,1]$ be a Markov kernel such that $\mathcal Q(x,\cdot)$ is absolutely continuous with non-zero density for all $x\in\R^d$ and let $\hat \rho\in\mathscr{P}(\R^d)$ be an absolutely continuous probability measure, which we call the target distribution. As before, we denote the densities of $\mathcal Q(x,\cdot)$ and $\hat \rho$ again by $\mathcal Q(x,y)$ and $\hat \rho(x)$.
Then the MH algorithm generates a sequence $(x_n)_n$ starting at $x_0\in\R^d$ using the following steps
\begin{enumerate}
    \item Draw $x'$ from $\mathcal Q(x_n,\cdot)$ and $u$ uniformly in $[0,1]$.
    \item Compute the acceptance probability $A(x_n,x')\coloneqq\min\left\{1,\frac{\hat\rho(x')\mathcal Q(x',x_n)}{\hat \rho(x_n)\mathcal Q(x_n,x')}\right\}.$
    \item Set
    $$
    x_{n+1}\coloneqq 
    \begin{cases}
    x',&\text{if }u<A(x_n,x'),\\
    x_n,&\text{otherwise.}
    \end{cases}
    $$
\end{enumerate}
The MH algorithm generates a time-homogeneous Markov chain with transition kernel $\mathcal K_\mathrm{MH}\colon\R^d\times\mathcal B(\R^d)\to[0,1]$ given by
$$
\mathcal K_\mathrm{MH}(x,B)=\int_A\mathcal Q(x,y)A(x,y)\d y+\delta_x(B)\int_{\R^d}\mathcal Q(x,y)(1-A(x,y))\d y.
$$
It can be shown that for any choice of $\mathcal Q$, the target distribution $\hat \rho$ is a stationary distribution of this Markov kernel in the sense that applying the kernel onto $\hat \rho$ yields again $\hat\rho$ or more precisely
$$
\hat \rho(B)=\int \mathcal K_\mathrm{MH}(x,B)\d\hat \rho(x),\quad \text{for any Borel set }B\in\mathcal B(\R^d).
$$

\subsection{Application to the WFR--JKO Scheme}

After each WFR--JKO step we are given an ensemble of samples $x_1,...,x_n$ which follow the distribution $\rho_k(\R^d)^{-1}\rho_k$. Further, we have computed the total mass $\rho_k(\R^d)$. However, due to the resampling step, some of these samples are located at the exact same position. Therefore, we apply one step of the MH algorithm on these samples. More precisely, we generate a new ensemble $y_1,...,y_n$ by sampling $y_i$ from $\mathcal K_\mathrm{MH}(x_i,\cdot)$. As a target distribution we use $\hat \rho\coloneqq \rho_k(\R^d)^{-1}\rho_k$. Since we know that the $x_i$ already follow the distribution $\hat \rho$ and since this is an stationary distribution of $\mathcal K_\mathrm{MH}$, we obtain that also the ensemble $y_1,...,y_n$ follows $\hat \rho$.
It is left to choose the Markov kernel $\mathcal Q$ from the MH algorithm. Here we consider the following choices.
\begin{enumerate}
    \item \textbf{Gaussian proposal:} One of the most standard choices in the MH algorithm is 
    \begin{equation}\label{eq:Gauss_proposal}
    \mathcal Q_\mathrm{Gauss}(x,\cdot)=\mathcal N(x,\sigma^2I_d),
    \end{equation} 
    where $\mathcal N(m,\Sigma)$ denotes the Gaussian distribution with mean $m$ and covariance matrix $\Sigma$ and $\mathcal N(\cdot;m,\Sigma)$ the corresponding density. The kernel has the denisty $\mathcal Q_\mathrm{Gauss}(x,x')=\mathcal N(x';x,\sigma^2I_d)$, which is symmetric such that the acceptance probability simplifies to $A(x,x')=\min\left\{1,\frac{\rho_k(x')}{\rho_k(x)}\right\}$.
    \item \textbf{Langevin proposal:} The MH algorithm is often combined with a Langevin sampling algorithm. This corresponds to applying one gradient descent step on the energy $-\log\hat\rho$ before adding Gaussian noise. 
    Since computing gradients of $-\log\hat\rho=-\log \rho_k$ is computational expensive, we replace the gradient of $-\log \rho_k$ here by $-\log \pi=\nabla V$, which can be computed directly and is motivated by the fact that $\rho_k$ is already close to $\pi$ for large $k$. This leads to the proposal distribution
    \begin{equation}\label{eq:Langevin_proposal}
    \mathcal Q_\mathrm{Langevin}(x,\cdot)=\mathcal N(x-\tau\nabla V(x),\sigma^2 I_d).
    \end{equation}
    The step size $\tau$ and the noise parameter $\sigma$ are often chosen as $\sigma^2=2\tau$ which corresponds to one step of the unadjusted Langevin algorithm. However, depending how closely $\rho_k$ approximates $\pi$ already also a choice of $\sigma^2>\tau$ can be useful.
    We obtain the density 
    $$
    \mathcal Q_\mathrm{Langevin}(x,x')=\mathcal N(x';x-\tau\nabla V(x),\sigma^2 I_d).
    $$
    \item \textbf{Proposal in the latent space:} The MH algorithm often suffers from low acceptance rates, whenever the target distribution is multimodal or somehow degenerated. Since our WFR--JKO algorithm usually starts at a very regular distribution $\rho_0$, we follow the idea to transform the samples from $\rho_k$ back to the latent space, perform the MH algorithm in the latent space with Gaussian proposal and transform the sample back to the data space. To make this procedure precise, we define the flow ODE $\phi\colon[0,k]\times\R^d\to\R^d$ by $\partial_t \phi_{l+t}(l+t,x)=v_t^l(\phi_{k+t}(k+t,x))$ with $\phi(0,x)=x$, where $v_t^l$ is the optimal velocity from the WFR--JKO step connecting $\rho_l$ and $\rho_{l+1}$. Then, we define the proposal distribution as
    \begin{equation}\label{eq:gauss_in_latent}
    \mathcal Q_\mathrm{latent}(x,B)=\mathcal Q_\mathrm{base}(\phi_k^{-1}(x),\phi_k^{-1}(B)),
    \end{equation}
    where $\mathcal Q_\mathrm{base}$ is some base Markov kernel. In our numerics we consider the cases, where it is given by the Gaussian proposal \eqref{eq:Gauss_proposal} or Langevin proposal \eqref{eq:Langevin_proposal} (with the Gaussian energy instead of $V$).
    We compute the density of $\mathcal Q_\mathrm{latent}$ in the next lemma.

\begin{lemma}
Let $v_t^l\colon[0,1]\times\R^d\to\R^d$ and $\partial_t \phi_{l+t}(x)=v_t^{l}\circ\phi_{l+t}(x)$ for $l=0,...,k-1$ and $t\in[0,1]$. Then, the transition kernel \eqref{eq:gauss_in_latent} has the density
$$
\mathcal Q_\mathrm{latent}(x,x')=\mathcal Q_\mathrm{base}(z,z')\exp\left(-\sum_{l=0}^{k-1}\int_0^1\div v_t^l(\phi_{l+t}(z'))\,\mathrm{d}t\right), 
$$
where $z'=\phi_k^{-1}(x')$ and $ z=\phi_k^{-1}(x)$.
\end{lemma}
\begin{proof}
It holds
\begin{align}
\mathcal Q_\mathrm{latent}(x,B)=\mathcal Q_\mathrm{base}(z,\phi_k^{-1}(B))
&=\int_{\phi_k^{-1}(B)} \mathcal Q_\mathrm{base}(z,z')\d z'\\
&=\int_{B}\mathcal Q_\mathrm{base}(z,\phi_k^{-1}(x'))|\det \phi_k^{-1}(x')| \d x',
\end{align}
where we used the substitution $z'=\phi_k^{-1}(x')$.
In particular, $\mathcal Q_\mathrm{latent}(x,\cdot)$ has the density 
$$
\mathcal Q_\mathrm{latent}(x,x')=\mathcal Q_\mathrm{base}(z,\phi_k^{-1}(x'))|\det(\phi_k^{-1}(z'))|.
$$
where $z'=\phi_k^{-1}(x')$. Then the claim follows by \eqref{eq:det}.
\end{proof}
\end{enumerate}

\subsection{Density Evaluation of $\rho_k$}
For computing the acceptance probabilities in the previous subsection, we have to evaluate the densities $\rho_k$ at some arbitrary point $x\in\R^d$.
To this end we denote by $v^k$ and $g^k$ the optimal velocities and reactions from the WFR--JKO step connecting $\rho_k$ and $\rho_{k+1}$. We define the concatenations $\tilde v\colon [0,k]\times\R^d\to\R^d$ and $\tilde g\colon[0,k]\times\R^d\to\R$ for $t\in[k,k+1]$ by
$$
\tilde v_t(x)=v^k_{t-k}(x), \quad \tilde g_t(x)=g^k_{t-k}(x).
$$
Then, by definition, the unique solution of the continuity equation with reaction
$$
\partial_t \tilde \mu_t +\div (\tilde v_t\tilde \mu_t) = \tilde g_t\tilde \mu_t,\quad \tilde \mu_0=\rho_0
$$
fulfills $\tilde \mu_k=\rho_k$ for all $k$. Using the time transform $s=k-t$ we obtain that the time reversals $\mu_s=\tilde \mu_{k-s}$, $v_s= \tilde v_{k-s}$ and $g_s=\tilde g_{k-s}$ that
$$
\partial_s \mu_s - \div(v_{s}\mu_s)= -g_s \mu_s,\quad\text{with}\quad \mu_0=\rho_k\text{ and } \mu_k=\rho_0.
$$
In particular, we get by \eqref{eq:nu} for the reverse flow ODE 
\begin{equation*}
\partial_t \phi^{\leftarrow}(t,\cdot)=-v_t\circ\phi^{\leftarrow}(t,\cdot), \quad \phi^{\leftarrow}(0,\cdot)=\mathrm{id}
\end{equation*}
that
$$
\rho_0=\phi^{\leftarrow}(k,\cdot)_\#(w\rho_k),\quad w(x)=\exp\left(\int_0^k -g_s\circ\phi^{\leftarrow}(s,x)\d s\right).
$$
Using the change-of-variables formula for densities, this implies for $\phi^{\leftarrow}_s=\phi^{\leftarrow}(s,\cdot)$ that the corresponding densities fulfill
$\rho_0(\phi_k^{\leftarrow}(x))=\frac{\rho_k(x)}{\det(\nabla \phi_k(x))}$. Inserting \eqref{eq:det} this can be reformulated as
$$
\rho_k(x)=\rho_0(\phi_k^{\leftarrow}(x))\exp\left(\int_0^k g_s\circ\phi_{s}^{\leftarrow}(x)-\div v_s\circ\phi_s^{\leftarrow}(x)\d s\right).
$$

\section{Numerical Results}\label{sec:numerics}
In this section, we report our promising numerical results in comparison with existing algorithms. To this end, we first introduce the test distributions on which we evaluate our algorithm. Then, we specify the detailed implementation and comparisons. Finally, we present our results.

\subsection{Test Distributions}\label{sec:test_distributions}

Next, we implement our algorithms. To this end, we evaluate the WFR--JKO scheme on the same test distributions as in \cite{Hertrich2025Importance}:
\begin{itemize}
    \item[-] \textbf{Mustache:} 
    The two-dimensional log-density is given as $\log\mathcal{N}(0,\Sigma) \circ T$ with $\Sigma= [1, \sigma ; \sigma,1]$ and $T(x_1,x_2) = (x_1, (x_2-(x_1^2-1)^2))$. Note that $\mathrm{det}(\nabla T(x))=1$ for all $x$. In particular, we obtain directly by the transformation formula that the normalization constant is one.
    Depending on $\sigma\in[0,1)$ close to $1$, this probability distribution has very long and narrow tails making it hard for classical MCMC methods to sample them. In our experiments we use $\sigma=0.9$.
    \item[-] \textbf{Shifted 8 Modes:} A two-dimensional Gaussian mixture model with 8 equal weighted modes and covariance matrix $\num{1e-2} I$. The modes are placed in a circle with radius 1 and center $(-1,0)$. Due to the shifted center classical MCMC methods have difficulties to distribute the mass correctly onto the modes.
    \item[-] \textbf{Shifted 8 Peaky:} This is the same distribution as the shifted 8 Modes with the difference that we reduce the width of the modes to the covariance matrix $\num{5e-3} I$. Since the modes are disconnected, it becomes harder to sample from them.
    \item[-] \textbf{GMM-$d$:} A $d$-dimensional Gaussian mixture model with 10 equal weighted modes with covariance matrix $\num{1e-2} I$ and means drawn randomly from a uniform distribution on $[-1,1]^d$. This leads to a peaky high-dimensional and multimodal probability distribution which is hard to sample from.
\end{itemize}

\subsection{Implementation and Hyperparameters}

In practice, we start with a very small initial step size $h_0$ and then iteratively increase the step size in each step by choosing $h_{k+1}=\tau h_k$ for some growth constant $\tau>1$. In order to avoid overfitting and the accumulation of errors, we resample the ensemble of samples after each WFR JKO step. Moreover, we apply a Metropolis--Hastings step with Gaussian proposal after each JKO step.

For all test distributions, we start with a Gaussian distribution $\mathcal N(0,\sigma_0^2 I)$ for a problem dependent scaling parameter $\sigma_0$. To parameterize the velocity fields $v$ and reaction rates $g$, we use a ResNet with $6$ blocks, where each block consists out of a two-layer network with \texttt{SiLU} activation and $512$ hidden neurons.
In order to simulate and differentiate the ODE solutions, we use the \texttt{rk4} solver from \texttt{torchdiffeq} library \cite{C2018} with a step size of $0.1$. 
We report the problem dependent hyperparameters in Table \ref{table:hyper}.

\begin{table*}[htbp]
\centering
\scalebox{.8}{%
\begin{tabular}{@{}lccccccc@{}}
\toprule
Distribution & $d$ & $h_0$ & $\tau$ & $\alpha$ & MH step size & $K$ & $\sigma_0$ \\
\midrule
Mustache &
2 & 0.0025 & 4.0 & 1.0 & 0.03 & 20 & 1.5 
\\

shifted 8 modes &
2 & 0.0025 & 4.0 & 1.0& 0.03 & 20 & 1.5 
\\

shifted 8 peaky &
2 & 0.0025 & 4.0 & 1.0 & 0.03 & 20 & 1.5 
\\

GMM-10 &
10 & 0.0025 & 4.0 & 1.0 & 0.03 & 20 & 1.5  \\

GMM-20 &
20 & 0.0025 & 4.0 & 1.0 & 0.03 & 10 & 1.5 \\

GMM-50 &
50 & 0.0025 & 4.0 & 1.0 & 0.03 & 10 & 1.5  \\

GMM-100 &
100 & 0.0025 & 4.0 & 1.0 & 0.03 & 10 & 1.5  \\

\bottomrule
\end{tabular}%
}
\caption{Hyperparameters for the WFR JKO algorithm: $d$ is the dimension of the distribution, $h_0$ is the initial step size in the WFR--JKO algorithm, $\tau$ is the growth of this step size, $\alpha$ is the parameter from the WFR metric, $K$ the number of WFR--JKO steps used, and $\sigma_0$ is the scaling of the initial distribution. In addition, we report the step sizes for solving the ODEs and for generating the proposals in the MH step.}

\label{table:hyper}
\end{table*}

\subsection{Comparisons}

We compare our methods with all comparison methods considered in \cite{Hertrich2025Importance}. Despite the neural Wasserstein JKO scheme (neural JKO) and its importance-corrected version (neural JKO IC) from \cite{Hertrich2025Importance} this includes classical Monte Carlo samplers like a Metropolis adjusted Langevin sampling (MALA) and Hamiltonian Monte Carlo (HMC), see e.g., \cite{B2017,RT1996} and two deep-learning based sampling algorithms, namely denoising diffusion samplers (DDS, \cite{VGD2023}) and continual repeated annealed flow transport Monte Carlo (CRAFT, \cite{MARD2022}).

\subsection{Results}

In order to evaluate the quality of the generated samples, we compute the energy distance \cite{S2002} between 50000 generated and ground truth samples, which corresponds to the maximum mean discrepancy with respect to the negative distance kernel $k(x,y)=-\|x-y\|$.
Given generated samples $\mathbf{X}=(X^1,...,X^N)$ and ground truth samples $\mathbf{Y}=(Y^1,...,Y^N)$, it is given by 
$$
\mathcal E(\mathbf{X},\mathbf{Y})=\frac{1}{N^2}\sum_{i,j=1}^N \|X^i-Y^j\|-\frac1{2N^2}\sum_{i,j=1}^N\|X^i-X^j\|-\frac1{2N^2}\sum_{i,j=1}^N\|Y^i-Y^j\|.
$$
The energy distance is a metric on the space of probability measures. In particular, a smaller energy distance indicates a better result.
The results are given in Table~\ref{table:comparisons}. Moreover, we plot two-dimensional marginals of generated samples in Figure~\ref{fig:table1}.

We observe that the WFR--JKO sampler is always either the best or second best method among the comparisons. Just the neural JKO IC method is competitive. This is expected, as the WFR JKO scheme formalizes the heuristic approach of ``transporting and reweighting'' from neural JKO IC and provides it with better theoretical guarantees. In addition, we note that the runtime of neural JKO IC depends exponentially on the number of resampling steps, see \cite[Remark 4.5]{Hertrich2025Importance}, while the WFR--JKO does not have this exponential behavior. All other methods degenerate for high dimensional multimodal target distributions.

\begin{table*}[htbp]
\centering
\resizebox{\textwidth}{!}{%
\begin{tabular}{@{}lcccccccc@{}}
\toprule
Distribution & MALA & HMC & DDS & CRAFT & Neural JKO
& Neural JKO IC & WFR JKO (ours) & Sampling Error \\
\midrule
Mustache
& $4.6 \times 10^{-2}$
& $1.7 \times 10^{-2}$
& $6.9 \times 10^{-2}$
& $9.2 \times 10^{-2}$
& $1.8 \times 10^{-2}$
& $\best 2.9 \times 10^{-3}$
& $\second 1.2 \times 10^{-2}$
& $8.6 \times 10^{-5}$ \\

shifted 8 modes
& $5.3 \times 10^{-3}$
& $4.1 \times 10^{-5}$
& $1.2 \times 10^{-2}$
& $5.2 \times 10^{-2}$
& $1.3 \times 10^{-1}$
& $\best 1.2 \times 10^{-5}$
& $\second 1.6 \times 10^{-5}$
& $2.6 \times 10^{-5}$ \\

shifted 8 peaky
& $1.3 \times 10^{-1}$
& $1.2 \times 10^{-1}$
& $1.1 \times 10^{-2}$
& $5.2 \times 10^{-2}$
& $1.3 \times 10^{-1}$
& $\second 3.4 \times 10^{-5}$
& $\best 3.1 \times 10^{-5}$
& $2.4 \times 10^{-5}$ \\

GMM-10
& $1.2 \times 10^{-2}$
& $1.2 \times 10^{-2}$
& $3.7 \times 10^{-3}$
& $1.8 \times 10^{-1}$
& $1.1 \times 10^{-2}$
& $\second 5.3 \times 10^{-5}$
& $\best 3.2 \times 10^{-5}$
& $4.6 \times 10^{-5}$ \\

GMM-20
& $9.1 \times 10^{-3}$
& $9.1 \times 10^{-3}$
& $5.0 \times 10^{-3}$
& $5.4 \times 10^{-1}$
& $1.0 \times 10^{-2}$
& $\second 1.1 \times 10^{-4}$
& $\best 1.0 \times 10^{-4}$
& $6.4 \times 10^{-5}$ \\

GMM-50
& $2.4 \times 10^{-2}$
& $2.4 \times 10^{-2}$
& $2.3 \times 10^{-2}$
& $1.8 \times 10^{0}$
& $2.7 \times 10^{-2}$
& $\best 1.0 \times 10^{-4}$
& $\second 4.1 \times 10^{-4}$
& $1.1 \times 10^{-4}$ \\

GMM-100
& $3.6 \times 10^{-2}$
& $3.7 \times 10^{-2}$
& $3.9 \times 10^{-2}$
& $2.8 \times 10^{1}$
& $4.7 \times 10^{-2}$
& $\best 6.0 \times 10^{-4}$
& $\second 1.7 \times 10^{-2}$
& $1.5 \times 10^{-4}$ \\

\bottomrule
\end{tabular}%
}
\caption{Energy distance between generated samples and ground truth samples. A smaller value indicates a better result. The best result for each test distribution is marked in green, the second-best is marked in blue.}
\label{table:comparisons}
\end{table*}

\begin{figure}[htbp]
\centering
\includegraphics[width=0.750\linewidth]{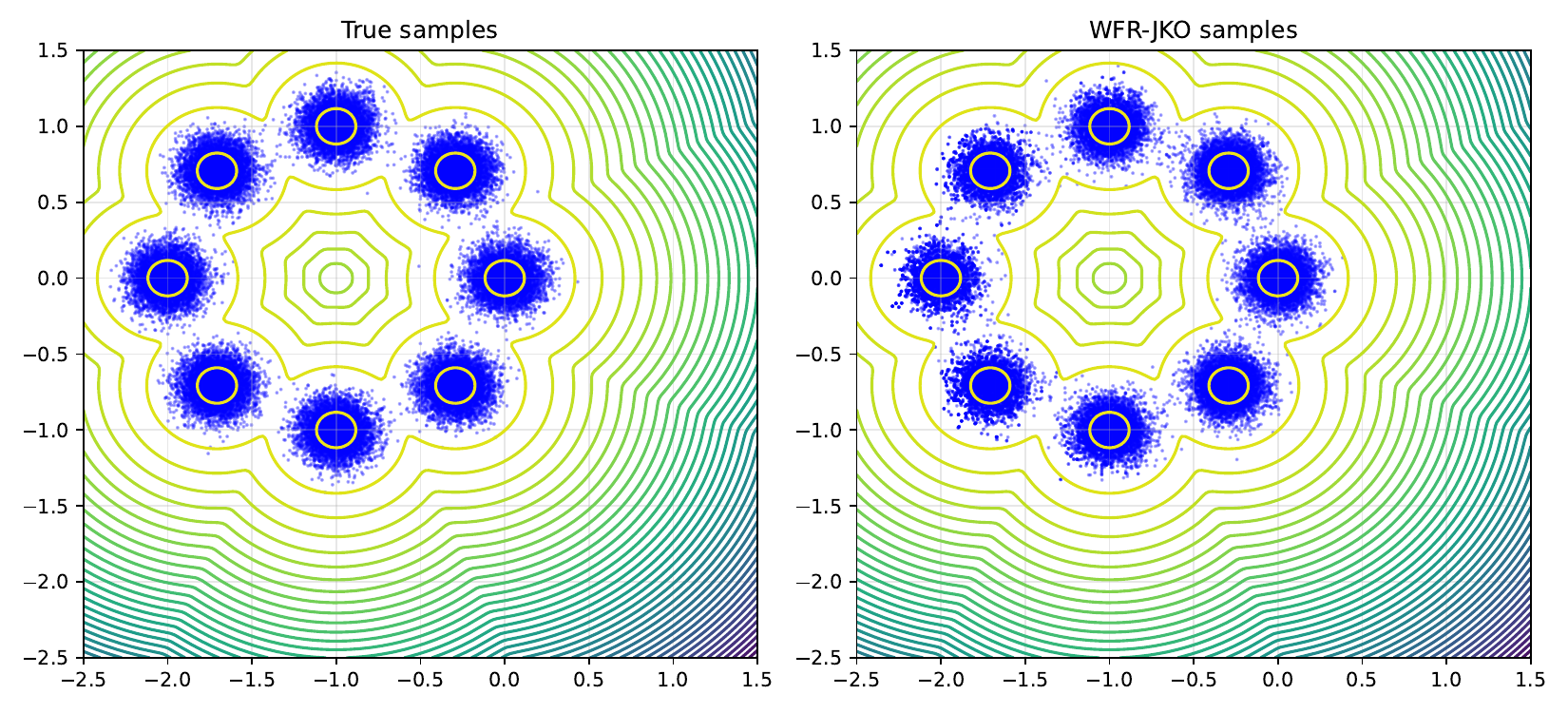}
\includegraphics[width=0.750\linewidth]{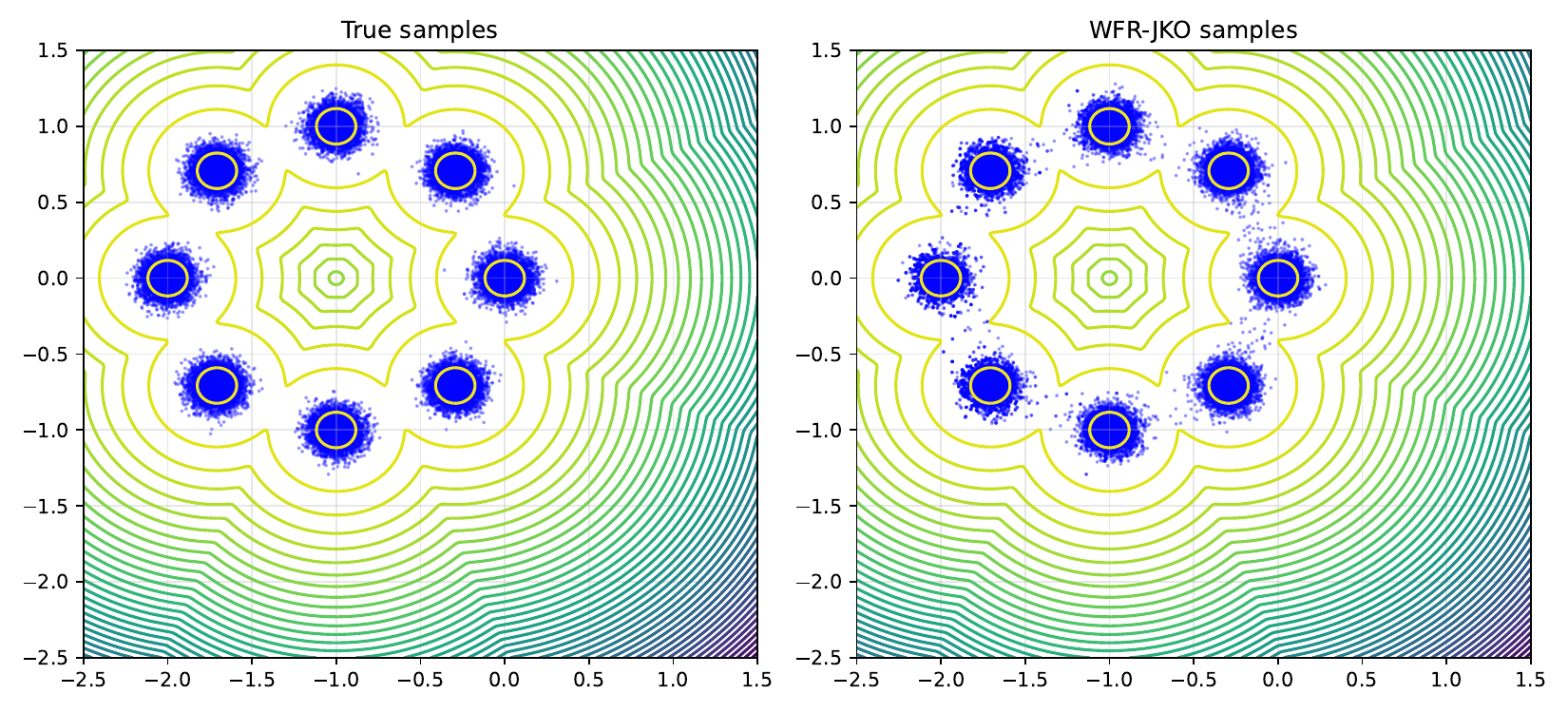}
\includegraphics[width=0.750\linewidth]{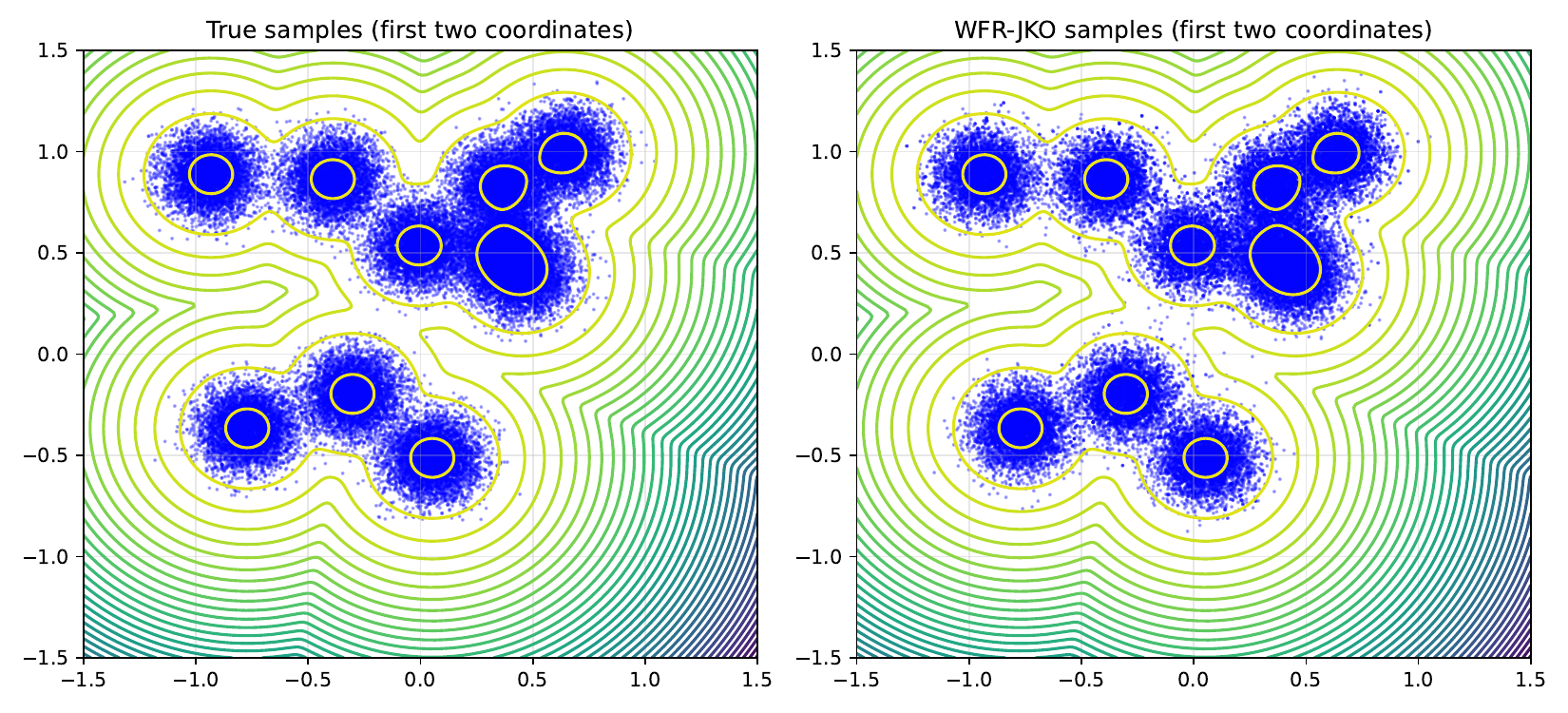}
\includegraphics[width=0.750\linewidth]{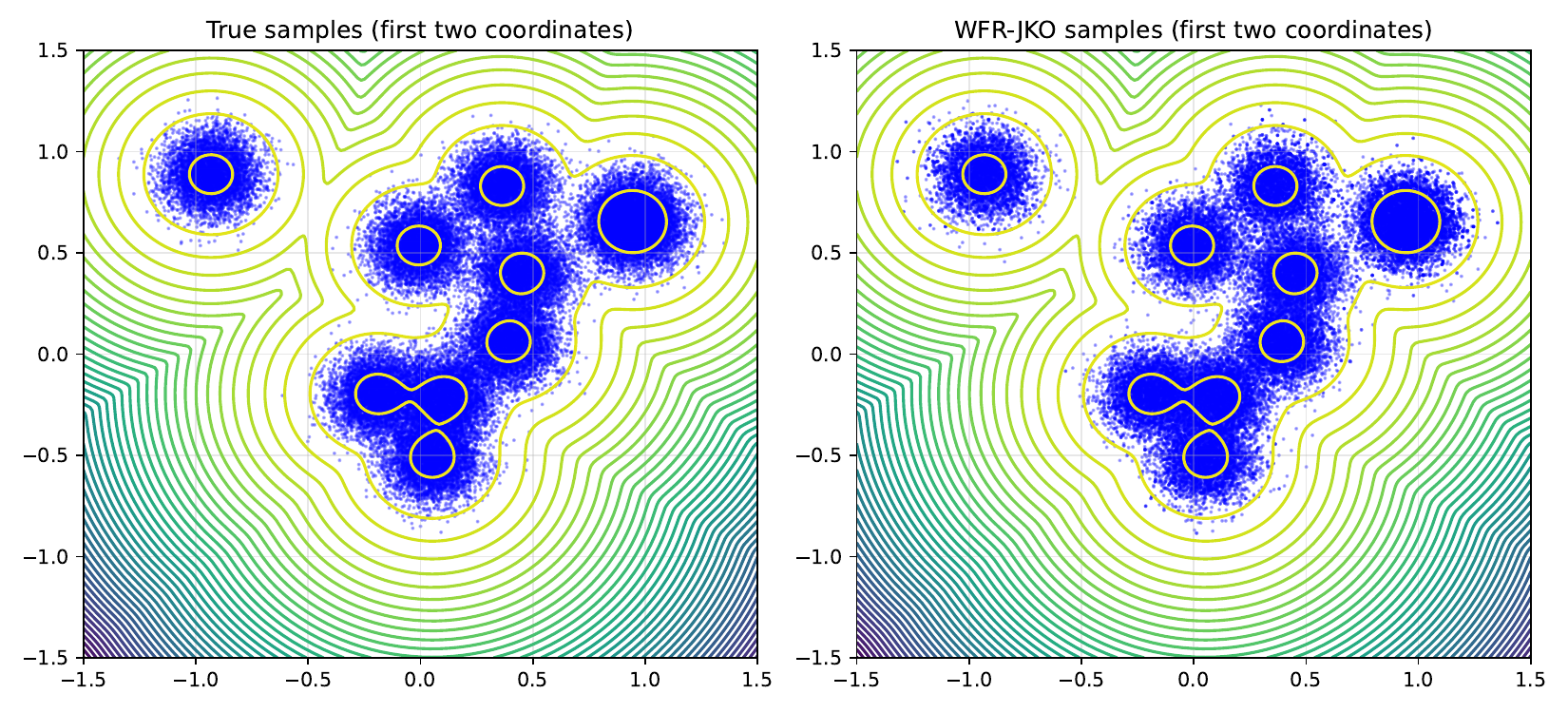}
\caption{
Plot of the first two coordinates of generated and ground truth samples for different examples.
\textbf{First line:} shifted 8 modes; \textbf{second line:} shifted 8 peaky; \textbf{third line:} GMM-10; \textbf{fourth line:} GMM-20.}
\label{fig:table1}
\end{figure}

\subsection*{LLM Usage}

We used Claude Fable for exploring the literature corresponding to continuous time gradient flows in the Wasserstein--Fisher--Rao metric. 
We used different LLMs for improving the implementations. 

\subsection*{Acknowledgments}
CD acknowledges funding by Deutsche Forschungsgemeinschaft (DFG, German Research Foundation) -- project number  442047500/SFB 1481 \emph{Sparsity and Singular Structures}. JH acknowledges funding by Deutsche Forschungsgemeinschaft (DFG, German Research Foundation), Project 572825596.
GS acknowledges funding by Deutsche Forschungsgemeinschaft (DFG, German Research Foundation),
Project STE 571/23-1.

\bibliography{reference}

\newpage
\appendix

\section{Auxiliary Proofs} \label{app_a}

\begin{proposition}
Let $r>0$ and $\pi\in\mathscr{M}(\R^d)$. Then, the set 
$\mathcal K_r \coloneqq \{\rho\in\mathscr{M}(\R^d):\kl(\rho\|\pi)\leq r\}$ is weakly sequentially compact.
\end{proposition}

\begin{proof} Recalling the convex duality
$x \log x -x+1 = \sup_{a \in \R} \{ax - e^a +1\}$, the KL divergence can be rewritten as
$$
\kl(\rho\|\pi) = \sup_{f \in C_b(\R^d)} \int f \, \d \rho - \int e^f - 1 \, \d \pi.
$$
Hence, we get, for $\rho\in\mathcal K_r$, that
\begin{equation}\label{A2}
\int_{\mathbb R^d} f\,d\rho
\leq
r+\int_{\mathbb R^d}(e^f-1)\,d\pi
\qquad\text{for every }f\in C_b(\mathbb R^d).
\end{equation}
Taking $f\equiv 1$  gives
\[
\rho(\mathbb R^d)
\leq
r+(e-1)\pi(\mathbb R^d),
\]
and consequently
$
\sup_{\rho\in\mathcal K_r}\rho(\mathbb R^d) \le C <\infty$.

Next, we show that $\mathcal K_c$ is uniformly tight. Let
$\varepsilon>0$ be arbitary small. Choose $a>0$ sufficiently large such that
\[
\frac{r}{a}<\frac{\varepsilon}{2}.
\]
Since $\pi$ is a finite measure on $\mathbb R^d$, it is tight. Thus,
there exists a compact set $K\subset\mathbb R^d$ such that
\[
\frac{e^a-1}{a}\pi(K^c)<\frac{\varepsilon}{2}.
\tag{A.4}
\]
Using the
measurable function $f=a\mathbf 1_{K^c}$, we obtain, for every
$\rho\in\mathcal K_r$, that
\[
r
\geq
a\rho(K^c)-(e^a-1)\pi(K^c).
\]
This implies
\[
\rho(K^c)
\leq
\frac{c}{a}
+
\frac{e^a-1}{a}\pi(K^c)
<\varepsilon
\]
and then
\[
\sup_{\rho\in\mathcal K_c}\rho(K^c)<\varepsilon,
\]
so that $\mathcal K_c$ is uniformly tight.
Now let $(\rho_n)_{n\geq1}$ be any sequence in $\mathcal K_c$. By
Prokhorov's theorem it follows that
there exists a subsequence, still denoted by $(\rho_n)$,
and some $\rho\in\mathcal M_+(\mathbb R^d)$ such that
$
\rho_n\rightharpoonup\rho
$.

It remains to verify that $\rho\in\mathcal K_r$. Since $\kl$ is
weakly lower semicontinuouswe obtain by 
$\rho_n\in\mathcal K_r$ for every $n$ that
\[
\operatorname{KL}(\rho\|\pi)
\leq
\liminf_{n\to\infty}
\operatorname{KL}(\rho_n\|\pi)
\leq r.
\]
Thus $\rho\in\mathcal K_r$ which finishes the proof.
\end{proof}

\section{Pseudo-code of the Algorithms}

We provide pseudo-code for the computations in Section \ref{sec:algorithms} in Algorithm \ref{alg:empirical:risk}, \ref{alg:run:jko:step} and \ref{alg:sample}.

\begin{algorithm}[htbp]\label{alg:empirical:risk}
\caption{Evaluate the empirical risk}
\KwIn{particles $X_{k}^{i}$ with log-density $\log\rho_k(X_k^i)$, velocity $v^k$, reaction $g^k$ and log-mass $\rho_{k}(\mathbb{R}^{d})$
}
\KwOut{The empirical risk $\widehat{\mathcal{L}}^{\wfr}_{k}(\theta)$}
Sample i.i.d. $d$-dimensional Rademacher variables $\xi_{1},\ldots,\xi_{N}\sim\mathrm{Rad}^{d}$. \\
Simulate the ODE system from $t=0$ to $t=1$ for $i=1,\ldots,N$:
\begin{align*}
\frac{\mathrm{d}}{\mathrm{d}t}\phi_{t}(X_{k}^{i}) &= v_{t}\circ\phi_{t}(X_{k}^{i}), \quad \phi_{0}(X_{k}^{i})=X_{k}^{i}, \\
\frac{\mathrm{d}}{\mathrm{d}t}\psi_{t}(X_{k}^{i}) &= g_{t}\circ\phi_{t}(X_{k}^{i}), \quad \psi_{0}(X_{k}^{i})=\log\rho_{k}(\mathbb{R}^{d}), \quad w_{t}(X_{k}^{i}) = \exp(\psi_{t}(X_{k}^{i})), \\
\frac{\mathrm{d}}{\mathrm{d}t}a_{t}(X_{k}^{i}) &= w_{t}(X_{k}^{i})\Bigl\{\|v_{t}\circ\phi_{t}(X_{k}^{i})\|_{2}^{2}+\frac{1}{\lambda}(g_{t}\circ\phi_{t}(X_{k}^{i}))^{2}\Bigr\}, \quad a_{0}(X_{k}^{i})=0, \\
\frac{\mathrm{d}}{\mathrm{d}t}\eta_{t}(X_{k}^{i})& = \xi_{i}^{\top}\nabla v_{t}\circ\phi_{t}(X_{k}^{i})\xi_{i}, \quad \eta_{t}(X_{k}^{i})=0.
\end{align*} \\
Compute the per-particle loss:
\begin{align*}
\widehat{\mathcal{L}}_{\wfr}^{k+1}(\theta;X_{k}^{i}) 
&\leftarrow \frac{1}{2h}a_{1}(X_{k}^{i})-w_{1}(X_{k}^{i})\eta_{1}(X_{k}^{i}) \\
&\quad +w_{1}(X_{k}^{i})\Bigl\{\log w_{1}(X_{k}^{i})+V\circ\phi_{1}(X_{k}^{i})+\log\rho_{k}(X_{k}^{i})-\log\rho_{k}(\mathbb{R}^{d})-1\Bigr\}.
\end{align*} \\
Compute the empirical risk 
$\widehat{\mathcal{L}}_{\wfr}^{k}(\theta) \leftarrow \frac{1}{N}\sum_{i=1}^{N}\widehat{\mathcal{L}}_{\wfr}^{k+1}(\theta;X_{k}^{i})$. \\
\Return{$\widehat{\mathcal{L}}_{\wfr}^{k+1}(\theta)$}
\end{algorithm}

\begin{algorithm}[htbp]\label{alg:run:jko:step}
\caption{Applying a JKO step}
\KwIn{particles $X_{k}^{i}$, log-densities $\log\rho_{k}(X_{k}^{1})$, velocity $v^k$ and reaction $g^k$}
\KwOut{weighted particles $Y_{k+1}^i$ with weights $w_{k+1}^i$ and log-densities $\rho_{k+1}(Y_{k+1}^i)$, log-mass $\log\rho_{k+1}(\mathbb{R}^{d})$}
Simulate the ODE system from $t=0$ to $t=1$ for $i=1,\ldots,N$:
\begin{align*}
\frac{\mathrm{d}}{\mathrm{d}t}\phi_{t}(X_{k}^{i}) &= v_{t}\circ\phi_{t}(X_{k}^{i}), \quad \phi_{0}(X_{k}^{i})=X_{k}^{i}, \\
\frac{\mathrm{d}}{\mathrm{d}t}\psi_{t}(X_{k}^{i}) &= g_{t}\circ\phi_{t}(X_{k}^{i}), \quad \psi_{0}(X_{k}^{i})=\log\rho_{k}(\mathbb{R}^{d}), \\
\frac{\mathrm{d}}{\mathrm{d}t}\zeta_{t}(X_{k}^{i}) &= -\div v_{t}\circ\phi_{t}(X_{k}^{i}), \quad \zeta_{0}(X_{k}^{i})=\log\rho_{k}(X_{k}^{i}),
\end{align*} \\
Update the particles $Y_{k+1}^{i}\leftarrow\phi_{1}(X_{k}^{i})$ for $i=1,\ldots,N$. \\
Evaluate the weights $w_{k+1}^{i} \leftarrow \exp(\psi_{1}(X_{k}^{i}))$ for $i=1,\ldots,N$. \\
Evaluate the log-density $\log\rho_{k+1}(Y_{k+1}^{i})\leftarrow \psi_{1}(X_{k}^{i})+\zeta_{1}(X_{k}^{i})$. \\
Evaluate the log-mass $\log\rho_{k+1}(\mathbb{R}^{d})\leftarrow \frac{1}{N}\sum_{i=1}^{N}w_{k+1}^{i}$. \\
\Return{samples $Y_{k+1}^{i}$, weights $w_i$, log-density $\log\rho_{k+1}(Y_{k+1}^{i}))$ and total mass $\rho_{k+1}(\mathbb{R}^{d})$.}
\end{algorithm}

\begin{algorithm}[htbp]\label{alg:sample}
\caption{Resampling}
\KwIn{weighted particles $Y_{k+1}^i$ with weights $w_{k+1}^i$ and log-densities $\rho_{k+1}(Y_{k+1}^i)$}
\KwOut{unweighted particles $X_{k+1}^i$ with log-densities $\rho_{k+1}(X_{k+1}^i)$}
Normalize the weights $(\widetilde{w}_{k+1}^{1},\ldots,\widetilde{w}_{k+1}^{N}) \leftarrow \mathrm{softmax}(w_{k+1}^{1},\ldots,w_{k+1}^{N})$. \\
Resample $(j_{1},\ldots,j_{N})\sim\mathrm{Multinormial}(\widetilde{w}_{k+1}^{1},\ldots,\widetilde{w}_{k+1}^{N})$. \\
$X_{k+1}^{i}\leftarrow X_{k+1}^{j_{i}}$ for $i=1,\ldots,N$. \\
$\log\rho_{k+1}(X_{k+1}^{i})\leftarrow \log\rho_{k+1}(X_{k+1}^{j_{i}})$ for $i=1,\ldots,N$. \\
\Return{$(X_{k+1}^{i},\log\rho_{k+1}(X_{k+1}^{i}))$ and $\rho_{k+1}(\mathbb{R}^{d})$.}
\end{algorithm}

\end{document}